\documentclass{siamltex}
\usepackage{amsfonts,amsmath,amssymb}
\usepackage{mathrsfs,mathtools}
\usepackage{enumerate}
\usepackage{xspace,MACROS/mydef}
\usepackage{hyperref}
\DeclareMathAlphabet{\mathpzc}{OT1}{pzc}{m}{it}

\newtheorem{remark}[theorem]{Remark}

\numberwithin{equation}{section}

\newcommand{\calL}{{\mathcal L}}
\newcommand{\calH}{{\mathcal H}}
\newcommand{\calE}{{\mathcal E}}
\newcommand{\calW}{{\mathcal W}}
\newcommand{\we}{\mathscr{W}}
\newcommand{\normC}[1]{ \|{#1}\|_{\C}}

\newcommand{\dt}{{\Delta t}}
\newcommand{\frakd}{{\mathfrak d}}
\newcommand{\frakc}{\mathfrak{c}}
\newcommand{\bmr}{{\boldsymbol r}}
\newcommand{\slope}{\mathfrak{s}}

\newcommand{\Ener}{\mathscr{E}}
\newcommand{\M}{\mathcal{M}}
\newcommand{\Vb}{\mathbf{V}}
\newcommand{\Ub}{\mathbf{U}}
\newcommand{\etab}{\boldsymbol{\eta}}
\newcommand{\fb}{\mathbf{f}}
\DeclareMathOperator*{\esssup}{esssup}

\title{A PDE approach to fractional diffusion: a space--fractional wave equation\thanks{EO is partially supported by CONICYT through FONDECYT project 3160201.}}

\author{Lehel Banjai\thanks{Maxwell Institute for Mathematical Sciences, School of Mathematical  \& Computer Sciences, Heriot-Watt University, Edinburgh EH14 4AS, UK. \texttt{l.banjai@hw.ac.uk}.}
\and
Enrique Ot\'arola\thanks{Departamento de Matem\'atica, Universidad T\'ecnica Federico Santa Mar\'ia, Valpara\'iso, Chile. \texttt{enrique.otarola@usm.cl}.}
}

\date{Draft version of \today.}

\begin{document}

\maketitle
\begin{abstract}
We study solution techniques for an evolution equation involving second order derivative in time and the spectral fractional powers, of order $s \in (0,1)$, of symmetric, coercive, linear, elliptic, second--order operators in bounded domains $\Omega$. We realize fractional diffusion as the Dirichlet-to-Neumann map for a nonuniformly elliptic problem posed on the semi--infinite cylinder $\C = \Omega \times (0,\infty)$. We thus rewrite our evolution problem as a quasi--stationary elliptic problem with a dynamic boundary condition and derive space, time, and space--time regularity estimates for its solution. The latter problem exhibits an exponential decay in the extended dimension and thus suggests a truncation that is suitable for numerical approximation. We propose and analyze two fully discrete schemes. The discretization in time is based on finite difference discretization techniques: trapezoidal and leapfrog schemes. The discretization in space relies on the tensorization of a first--degree FEM in $\Omega$ with a suitable $hp$--FEM in the extended variable. For both schemes we derive stability and error estimates.
\end{abstract}

\begin{keywords}
space--fractional wave equation, fractional diffusion, nonlocal operators, regularity estimates, fully discrete methods, finite elements, stability, error estimates.
\end{keywords}

\begin{AMS}
26A33,  
35J70, 	
35R11,  
65M12,  
65M15.  
65M60.	
\end{AMS}

\section{Introduction}
\label{sec:introduccion}

We are interested in the numerical approximation of an initial boundary value problem for a space--fractional wave equation. Let $\Omega$ be an open and bounded domain in $\R^n$ ($n\ge1$) with boundary $\partial\Omega$. Given $s \in (0,1)$, a forcing function $f$, and initial data $g$ and $h$, we seek $u$ such that
\begin{equation}
\label{eq:fractional_wave}
\begin{dcases}
  \partial_{t}^2 u + \mathcal{L}^s u  = f \ \text{ in } \Omega\times(0,T),
  \\
  u(0)  = g \ \text{ in } \Omega, \quad \partial_t u(0) = h \ \text{ in } \Omega.
\end{dcases}
\end{equation}
Here, $\mathcal{L}$ denotes the linear, elliptic, self--adjoint, second order, differential operator
\begin{equation*}
\label{second_order}
 \mathcal{L} w = - \DIV_{x'} (A \GRAD_{x'} w ) + c w,
\end{equation*}
supplemented with homogeneous Dirichlet boundary conditions. The coefficient $A \in C^{0,1}(\Omega,\GL(n,\R))$ is symmetric and uniformly positive definite and $0 \leq c \in L^\infty(\Omega)$. By $\calLs$, with $s \in (0,1)$, we denote the \emph{spectral} fractional powers of the operator $\calL$.

One of the most common and studied physical processes is diffusion: the tendency of a substance to evenly spread into an available space. Classical models of diffusion lead to well--known models and even better studied equations. However, in recent times, it has become evident that many of the assumptions that lead to these models are not always satisfactory or even realistic: memory, heterogeneity or a multiscale structure might violate them. In this setting, the assumption of locality does not hold and to describe diffusion one needs to resort to nonlocal operators. Different models of diffusion have been proposed, fractional diffusion being one of them. An incomplete list of problems where fractional diffusion appears includes finance \cite{MR2064019,PH:97,Carr.Geman.ea2002}, turbulent flow \cite{chen23}, quasi--geostrophic flows models \cite{CV:10,Kiselev2007}, models of anomalous thermoviscous behaviors \cite{doi:10.1121/1.1646399}, peridynamics \cite{MR3023366,MR3618711}, and image science \cite{doi:10.1137/070698592,Gatto2015}.

The design of efficient solution techniques for problems involving fractional diffusion is intricate, mainly due to the nonlocal character of $\calLs$ \cite{CS:11,CT:10,CS:07,CDDS:11}. Recently, and in order to overcome such a nonlocal feature, it is has been proved useful in numerical analysis the application of the Caffarelli--Silvestre extension \cite{BMNOSS:17,NOS,NOS3}. When $\mathcal{L} = -\Delta$ and $\Omega = \R^n$, \ie in the case of the Laplacian in the whole space, Caffarelli and Silvestre \cite{CS:07} showed that $\mathcal{L}^s$ can be realized as the Dirichlet-to-Neumann map for an extension problem to the upper half--space $\R_{+}^{n+1}$; the extension corresponds to a nonuniformly elliptic PDE. This result was later extended in \cite{CT:10,CDDS:11,ST:10} to bounded domains $\Omega$ and more general operators, thereby obtaining an extension problem posed on the semi--infinite cylinder $\C:= \Omega \times (0,\infty)$. We shall thus rewrite problem \eqref{eq:fractional_wave} as the following quasi--stationary elliptic problem with a dynamic boundary condition \cite{MR2600998,MR2737788,MR2954615,MR3192423}:
\begin{equation}
\label{eq:wave_alpha_extension}
\begin{dcases}
-\DIV \left( y^{\alpha} \mathbf{A} \GRAD \ue \right) + y^{\alpha} c\ue = 0 \textrm{ in } \C \times(0,T), & \ue = 0 \textrm{ on }\partial_L \C  \times (0,T),\\
d_s \partial_{t}^2 \ue + \partial_{\nu}^{\alpha} \ue  = d_s f \textrm{ on } (\Omega \times \{ 0\}) \times (0,T),
\end{dcases}
\end{equation}
with the initial conditions
\begin{equation}
\label{eq:initial_cond}
 \ue = g \textrm{ on } \Omega \times \{ 0\}, ~t=0, \quad \partial_t \ue = h \textrm{ on } \Omega \times \{ 0\}, ~t=0,
\end{equation}
where $\partial_L \C= \partial \Omega \times [0,\infty)$ corresponds to the lateral boundary of $\C$, $\alpha =1-2s \in (-1,1)$, $d_s=2^\alpha \Gamma(1-s)/\Gamma(s)$ and the conormal exterior derivative of $\ue$ at $\Omega \times \{ 0 \}$ is
\begin{equation}
\label{eq:conormal_derivative}
\partial_{\nu}^{\alpha} \ue = -\lim_{y \rightarrow 0^+} y^\alpha \ue_y;
\end{equation}
the limit must be understood in the sense of distributions \cite{CS:07,CDDS:11,ST:10}. We will call $y$ the \emph{extended variable} and the dimension $n+1$ in $\R_+^{n+1}$ the \emph{extended dimension} of problem \eqref{eq:wave_alpha_extension}--\eqref{eq:initial_cond}. Finally, $\mathbf{A} =  \diag \{A,1\}  \in C^{0,1}(\C,\GL(n+1,\R))$. With the solution $\ue$ to the extension problem \eqref{eq:wave_alpha_extension}--\eqref{eq:initial_cond} at hand, we can find the solution to \eqref{eq:fractional_wave} via \cite{MR2600998,CT:10,CS:07,CDDS:11,MR2737788,MR2954615,ST:10,MR3192423}:
\[
 u = \ue|_{y=0}.
\]

To the best of the authors’ knowledge this is the first work that provides a comprehensive treatment of efficient solution techniques for the space--fractional wave equation \eqref{eq:fractional_wave}. In \eqref{eq:fractional_wave}, $\mathcal{L}^s$ denotes the \emph{spectral} fractional powers of the operator $\mathcal{L}$. Recently, problem \eqref{eq:fractional_wave} has been considered in the literature but with 
$\mathcal{L}^s = (-\Delta)^s$ being the \emph{integral} fractional Laplace operator \cite{Landkof}: the work \cite{ABB2} proposes a discrete scheme that is based on standard Galerkin finite elements for space discretization and the convolution quadrature approach for the discretization in time. We immediately comment that the \emph{spectral} and \emph{integral} definitions of the fractional Laplace operator do not coincide. In fact, as shown in \cite{MR3246044} their difference is positive and positivity preserving. This, in particular, implies that the boundary behavior of the respective solutions is quite different \cite{MR3489634,Grubb}.

The outline of this paper is as follows. In section \ref{sec:Prelim} we introduce some terminology used throughout this work. We recall the definition of the fractional powers of elliptic operators via spectral theory in section \ref{sub:fractional_L}, and in section \ref{sub:CaffarelliSilvestre} we briefly describe their localization via the Caffarelli--Silvestre extension and introduce the functional framework that is suitable for studying problem \eqref{eq:wave_alpha_extension}. In section \ref{sec:wellposedness}, we review existence and uniqueness results together with energy--estimates for problems \eqref{eq:fractional_wave} and \eqref{eq:wave_alpha_extension}. In section \ref{sec:Regularity} we present space, time and space--time regularity results for the solution of problem \eqref{eq:wave_alpha_extension}. The numerical analysis for problem \eqref{eq:wave_alpha_extension} begins in section \ref{sec:truncation} where we introduce a truncated problem on the bounded cylinder $\C_{\Y} = \Omega \times (0,\Y)$ and study some properties of its solution. In section \ref{sec:space_time_discretization} we preset two fully discrete schemes for the truncated version of \eqref{eq:wave_alpha_extension} studied in  section \ref{sec:truncation}. Both of them are based on the scheme of \cite{BMNOSS:17} for space discretization. For time discretization we consider an implicit finite difference discretization scheme and the so--called leapfrog scheme. We derive stability and a priori error estimates for the proposed schemes for all $s \in (0,1)$. In section~\ref{sec:numerics} we comment on some implementation details pertinent to the problem at hand and present numerical experiments in 1D and 2D.

\section{Notation and preliminaries}
\label{sec:Prelim}

Throughout this work $\Omega$ is an open, bounded, and convex polytopal subset of $\R^n$ ($n\geq1$) with boundary $\partial\Omega$. We define the semi--infinite cylinder $\C := \ \Omega \times (0,\infty)$ and its lateral boundary $\partial_L \C  := \partial \Omega \times [0,\infty)$. For $\Y>0$, we define the truncated cylinder with base $\Omega$ and height $\Y$ as $\C_\Y := \Omega \times (0, \Y)$; its lateral boundary is denoted by $\partial_L \C_{\Y}  = \partial \Omega \times (0,\Y)$. If $x \in \R^{n+1}$, we write 
$
x =  (x',y),
$
with $x' \in \R^n$ and $y\in\R$. 

For an open set $D \subset \R^n$ ($n \geq 1$), if $\omega$ is a weight and $p \in (1,\infty)$, we denote the Lebesgue space of $p$-integrable functions with respect to the measure $\omega \diff x$ by $L^p(\omega,D)$ \cite{HKM,Kufner80,Turesson}. Similar notation will be used for weighted Sobolev spaces. If $T >0$ and $\phi: D \times(0,T) \to \R$, we consider $\phi$ as a function of $t$ with values in a Banach space $X$,
$
 \phi:(0,T) \ni t \mapsto  \phi(t) \equiv \phi(\cdot,t) \in X
$.
For $1 \leq p \leq \infty$, $L^p( 0,T; X)$ is the space of $X$-valued functions whose norm in $X$ is in $L^p(0,T)$. This is a Banach space for the norm
\[
  \| \phi \|_{L^p( 0,T;X)} = \left( \int_0^T \| \phi(t) \|^p_X \diff t \right)^{\hspace{-0.1cm}\frac{1}{p}} 
  , \quad 1 \leq p < \infty, \quad
  \| \phi \|_{L^\infty( 0,T;X)} = \esssup_{t \in (0,T)} \| \phi(t) \|_X.
\]

Whenever $X$ is a normed space, $X'$ denotes its dual and $\|\cdot\|_{X}$ its norm. If, in addition, $Y$ is a normed space, we write $X \hookrightarrow Y$ to indicate continuous embedding.
The relation $a \lesssim b$ means $a \leq Cb$, with a constant $C$ that neither depends on $a$ or $b$. The value of $C$ might change at each occurrence.

The next result, that follows from Young's inequality for convolutions, will be instrumental in the analysis that we will perform.
\begin{lemma}[continuity]
\label{le:continuity}
If $g \in L^2(0,T)$ and $\phi \in L^1(0,T)$, then the operator
\[
 g \mapsto \Phi, \qquad \Phi(t) = \phi \star g (t) = \int_0^t \phi(t-r) g(r) \diff r
\]
is continuous from $L^2(0,T)$ into itself and 
$
 \| \Phi \|_{L^2(0,T)} \leq \| \phi \|_{L^1(0,T)} \| g \|_{L^2(0,T)}.
$
\end{lemma}

Finally, since we assume $\Omega$ to be convex, in what follows we will make use, without explicit mention, of the following regularity result \cite{Grisvard}:
\begin{equation}
\label{eq:Omega_regular}
\| w \|_{H^2(\Omega)} \lesssim \| \mathcal{L} w \|_{L^2(\Omega)} \quad \forall w \in H^2(\Omega) \cap H^1_0(\Omega).
\end{equation}


\subsection{Fractional powers of second order elliptic operators}
\label{sub:fractional_L}

We adopt the \emph{spectral} definition for the fractional powers of the operator $\mathcal{L}$. To define  $\mathcal{L}^s$, we begin by noticing that $\calL$ induces the following inner product 
$a_{\Omega}(\cdot,\cdot)$ on $H^1_0(\Omega)$ 
\begin{equation}
\label{eq:blfOmega} 
a_{\Omega}(w,v) = \int_\Omega \left( A \GRAD w \cdot \GRAD v + c w v \right) \diff x',
\end{equation} 
and that $\calL: H^1_0(\Omega) \ni u \mapsto a_{\Omega}(u,\cdot) \in H^{-1}(\Omega)$ is an isomorphism. The eigenvalue problem:
\begin{equation}
\label{eq:eigenpairs}
(\lambda,\varphi) \in \R \times H_0^1(\Omega) \setminus \{ 0\}: \quad    a_{\Omega}( \varphi,v) = \lambda (\varphi,v)_{L^2(\Omega)} \quad \forall v \in H_0^1(\Omega)
\end{equation}
has a countable collection of solutions $\{ \lambda_{\ell}, \varphi_{\ell} \}_{\ell\in \mathbb N} \subset \R_+ \times H_0^1(\Omega)$ with the real eigenvalues enumerated in increasing order, counting multiplicities, and such that, $\{\varphi_{\ell} \}_{\ell=1}^{\infty}$ is an orthonormal basis of $L^2(\Omega)$ and an orthogonal basis of $(H_0^1(\Omega),a_{\Omega}(\cdot,\cdot))$ \cite{BS,Kato}. With these eigenpairs at hand, we introduce, for $s \geq 0$, the fractional Sobolev space
\begin{equation}
\label{def:Hs}
  \Hs = \left\{ w = \sum_{\ell=1}^\infty w_\ell \varphi_\ell: \| w \|^2_{\Hs}:= \sum_{\ell=1}^\infty \lambda_\ell^s |w_\ell|^2 < \infty \right\},
\end{equation}
and denote by $\Hsd$ the dual space of $\Hs$. The duality pairing between the aforementioned spaces will be denoted by $\langle \cdot, \cdot \rangle$. We notice that, if  $s \in (0,\tfrac12)$, $\Hs = H^s(\Omega) = H_0^s(\Omega)$, while, for $s \in (\tfrac12,1)$, $\Hs$ can be characterized by \cite{Lions,McLean,Tartar}
\[
  \Hs = \left\{ w \in H^s(\Omega): w = 0 \text{ on } \partial\Omega \right\}.
\]
If $s = \frac{1}{2}$, we have that $\mathbb{H}^{\frac{1}{2}}(\Omega)$ is the so--called Lions--Magenes space $H_{00}^{\frac{1}{2}}(\Omega)$ \cite{Lions,Tartar}. If $s\in(1,2]$, owing to \eqref{eq:Omega_regular}, we have that $\Hs = H^s(\Omega)\cap H^1_0(\Omega)$ \cite{ShinChan}.

The fractional powers of the operator $\mathcal L$ are thus defined by
\begin{equation}
  \label{def:second_frac}
 \mathcal{L}^s: \Hs \to \Hsd, \quad  \mathcal{L}^s w  := \sum_{\ell=1}^\infty \lambda_\ell^{s} w_\ell \varphi_\ell,  \quad s \in (0,1).
\end{equation} 

\subsection{Weighted Sobolev spaces}
\label{sub:CaffarelliSilvestre}
Both extensions, the one by Caffarelli and Silvestre \cite{CS:07} and the ones in \cite{CT:10,CDDS:11,ST:10} for bounded domains $\Omega$ and general elliptic operators, require us to deal with a local but nonuniformly elliptic problem. To provide an analysis for the latter it is thus suitable to define the weighted space
\begin{equation}
  \label{HL10}
  \HL(y^{\alpha},\C) = \left\{ w \in H^1(y^\alpha,\C): w = 0 \textrm{ on } \partial_L \C\right\}.
\end{equation}
Since $\alpha \in (-1,1)$, $|y|^\alpha$ belongs to the Muckenhoupt class $A_2$ \cite{MR1800316,Muckenhoupt,Turesson}. The following important consequences thus follow immediately: $H^1(y^{\alpha},\C)$ is a Hilbert space and $C^{\infty}(\Omega) \cap H^1(y^{\alpha},\C)$ is dense in $H^1(|y|^{\alpha},\C)$ (cf.~\cite[Proposition 2.1.2, Corollary 2.1.6]{Turesson}, \cite{KO84} and \cite[Theorem~1]{GU}). In addition, as \cite[inequality (2.21)]{NOS} shows, the following \emph{weighted Poincar\'e inequality} holds:
\begin{equation}
\label{Poincare_ineq}
\| w \|_{L^2(y^{\alpha},\C)} \lesssim \| \GRAD w \|_{L^2(y^{\alpha},\C)} \quad \forall w \in \HL(y^{\alpha},\C).
\end{equation}
Thus, $\| \GRAD w \|_{L^2(y^{\alpha},\C)}$ is equivalent to the norm in $\HL(y^{\alpha},\C)$. 

We define the bilinear form $a: \HL(y^{\alpha},\C) \times \HL(y^{\alpha},\C) \rightarrow \mathbb{R}$ by
\begin{equation}
\label{a}
a(w,\phi) :=   \frac{1}{d_s}\int_{\C} y^{\alpha} \left( \mathbf{A}(x) \GRAD w \cdot \GRAD \phi + c(x') w \phi \right) \diff x,
\end{equation}
which is continuous and, owing to \eqref{Poincare_ineq}, coercive on $\HL(y^{\alpha},\C)$. Consequently, it induces an inner product on $\HL(y^{\alpha},\C)$ and the following energy norm:
\begin{equation}
\label{eq:norm-C} 
\normC{w}^2:= a(w,w) \sim \|\nabla w\|^2_{L^2(y^\alpha,\C)}.
\end{equation}

For $w \in H^1( y^{\alpha},\C)$, $\tr w$ denotes its trace onto $\Omega \times \{ 0 \}$. We recall that, for $\alpha = 1-2s$, \cite[Proposition 2.5]{NOS} yields
\begin{equation}
\label{Trace_estimate}
\tr \HL(y^\alpha,\C) = \Hs,
\qquad
  \|\tr w\|_{\Hs} \leq C_{\mathrm{tr}} \normC{w}.
\end{equation}

The seminal work of Caffarelli and Silvestre \cite{CS:07} and its extensions to bounded domains \cite{CT:10,CDDS:11,ST:10} showed that the operator $\mathcal{L}^s$ can be realized as the Dirichlet-to-Neumann map for a nonuniformly elliptic boundary value problem. Namely, if $\mathfrak{U}$ solves
\begin{equation}
\label{eq:extension}
    -\DIV\left( y^\alpha \mathbf{A} \GRAD \mathfrak{U} \right) + c y^\alpha  \mathfrak{U}= 0  \text{ in } \C, \quad
    \mathfrak{U}= 0  \text{ on } \partial_L \C, \quad
    \partial_\nu^\alpha \mathfrak{U} = d_s f  \text{ on } \Omega \times \{0\},
\end{equation}
where $\alpha = 1-2s$, $\partial_\nu^\alpha \mathfrak{U} = -\lim_{y\downarrow 0} y^\alpha \mathfrak{U}_y$ and $d_s = 2^\alpha \Gamma(1-s)/\Gamma(s)$ is a normalization constant, then $\mathfrak{u} = \tr \mathfrak{U} \in \Hs$ solves
\begin{equation}
\label{eq:fraclap}
\mathcal{L}^s \mathfrak{u} = f.
\end{equation}

\section{Well--posedness and energy estimates} 
\label{sec:wellposedness}
In this section we briefly review the results of \cite{OS:17} regarding the existence and uniqueness of weak solutions for problems \eqref{eq:fractional_wave} and \eqref{eq:wave_alpha_extension}--\eqref{eq:initial_cond}. We also provide basic energy estimates.

\subsection{The fractional wave equation}
\label{sub:existunique}

We assume that the data of problem \eqref{eq:fractional_wave} is such that $f \in L^2(0,T;L^2(\Omega))$, $g \in \Hs$ and $h \in L^2(\Omega)$ and define
\begin{equation}
\label{eq:Lambda}
\Lambda(f,g,h):= \| f \|_{L^2(0,T;L^2(\Omega))} + \| g \|_{\Hs} + \| h \|_{L^2(\Omega)}.
\end{equation}

\begin{definition}[weak solution for \eqref{eq:fractional_wave}]
\label{def:weak_2}
We call $u \in L^2(0,T;\Hs)$, with $\partial_t u \in L^2(0,T;L^2(\Omega))$ and $\partial_{t}^2 u \in L^2(0,T;\Hsd)$, a weak solution of problem \eqref{eq:fractional_wave} if $u(0) = g$, $\partial_t u(0) = h$ and a.e. $t \in (0,T)$,
\[
  \langle \partial_t^2 u , v \rangle + \langle \mathcal{L}^s u , v \rangle
  = \langle f , v \rangle \quad \forall v \in \Hs,
\]
where $\langle \cdot , \cdot \rangle$ denotes the the duality pairing between $\Hs$ and $\Hsd$.
\end{definition}

The following remark is in order.

\begin{remark}[initial conditions]\rm
Since a weak solution $u$ of \eqref{eq:fractional_wave} satisfies that $u \in L^2(0,T;\Hs)$, $\partial_t u \in L^2(0,T;L^2(\Omega))$ and $\partial_{t}^2 u \in L^2(0,T;\Hsd)$, an application of \cite[Lemma 7.3]{MR3014456} reveals that $u \in C([0,T]; L^2(\Omega))$ and $\partial_t u \in C([0,T]; \Hsd)$. The initial conditions involved in Definition \ref{def:weak_2} are thus appropriately defined.
\label{rem:initial_cond_2}
\end{remark}

\begin{theorem}[well--posedness]
Given $s \in (0,1)$, $f \in L^2(0,T;L^2(\Omega))$, $g \in \Hs$ and $h \in L^2(\Omega)$, problem \eqref{eq:fractional_wave} has a unique weak solution. In addition, 
\begin{equation}
\label{eq:energy_estimate_2}
\| u\|_{L^{\infty}(0,T;\Hs)} + \| \partial_t u\|_{L^{\infty}(0,T;L^2(\Omega))} \lesssim \Lambda(f,g,h),
\end{equation}
where the hidden constant is independent of the problem data.
\label{thm:wp_2}
\end{theorem}
\begin{proof}
The desired results can be obtained by slightly modifying the arguments, based on a Galerkin technique, of \cite{Evans,Lions,MR3014456}. 
\end{proof}

\subsection{The extended fractional wave equation}
\label{sub:extended}

We consider the following notion of weak solution for problem \eqref{eq:wave_alpha_extension}--\eqref{eq:initial_cond}.

\begin{definition}[extended weak solution]
\label{def:weak_wave}
We call $\ue \in L^{\infty}(0,T;\HL(y^{\alpha},\C))$, with $\tr \partial_t \ue \in L^{\infty}(0,T;L^2(\Omega))$ and $\tr \partial_{t}^{2} \ue \in L^2(0,T;\Hsd)$, a weak solution of problem \eqref{eq:wave_alpha_extension}--\eqref{eq:initial_cond} if $\tr \ue(0) = g$, $\tr\partial_t \ue(0) = h$ and, for a.e. $t \in (0,T)$,
\begin{equation}
\label{eq:weak_wave}
  \langle \tr \partial_t^{2} \ue , \tr \phi \rangle + a(\ue,\phi) = \langle f , \tr \phi \rangle \quad \forall \phi \in \HL(y^\alpha,\C),
\end{equation}
where $\langle \cdot , \cdot \rangle$ denotes the the duality pairing between $\Hs$ and $\Hsd$ and the bilinear form $a$ is defined as in \eqref{a}.
\end{definition}

\begin{remark}[dynamic boundary condition] \rm
Problem \eqref{eq:weak_wave} is an elliptic problem with the following dynamic boundary condition: $\partial_{\nu}^{\alpha} \ue = d_s(f - \tr  \partial_t^{2} \ue)$ on $\Omega \times \{0\}$.
\end{remark}

We present the following important localization result \cite{MR2600998,MR3393253,CT:10,CS:07,CDDS:11,MR2737788,MR2954615,ST:10,MR3192423}.

\begin{theorem}[Caffarelli--Silvestre extension property]
\label{thm:CS_NOT_ST}
Let $s \in (0,1)$. If $f$, $g$ and $h$ are as in Theorem~\ref{thm:wp_2}, then the unique weak solution of problem \eqref{eq:fractional_wave}, in the sense of Definition \ref{def:weak_2} satisfies that $u = \tr \ue$, where $\ue$ denotes the unique weak solution to problem \eqref{eq:wave_alpha_extension} in the sense of Definition \ref{def:weak_wave}.
\end{theorem}

We now present the well–posedness of problem \eqref{eq:weak_wave} together with energy estimates for its solution.

\begin{theorem}[well--posedness] 
\label{thm:energy_wave_2}
Given $s \in (0,1)$, $f \in L^2(0,T;L^2(\Omega))$, $g \in \Hs$ and $h \in L^2(\Omega)$, then problem \eqref{eq:wave_alpha_extension}--\eqref{eq:initial_cond} has a unique weak solution in the sense of Definition \ref{def:weak_wave}. In addition,
\end{theorem}
\begin{equation}\label{eq:energy_estimate_wave_2}
\| \nabla \ue \|_{L^{\infty}(0,T;L^2(y^{\alpha},\C))} + \| \tr \partial_t \ue \|_{L^{\infty}(0,T;L^2(\Omega))} 
\lesssim \Lambda(f,g,h),
\end{equation}
where the hidden constant is independent of the problem data and $\Lambda(f,g,h)$ is defined as in \eqref{eq:Lambda}.
\begin{proof}
See \cite[Theorem 3.11]{OS:17}. 
\end{proof}

\begin{remark}[initial data]
\label{rem:initial_data} \rm
The initial data $g$ and $h$ of problem \eqref{eq:weak_wave} determine $\ue(0)$ and $\partial_t\ue(0)$ only on $\Omega \times \{ 0\}$ in a trace sense. However, in the analysis that follows it will be necessary to consider their extension to the whole cylinder $\C$. We thus define $\ue(0) = \mathcal{E}_{\alpha}g$ and $\ue_t(0) = \mathcal{E}_{\alpha}h$, where the $\alpha$--harmonic extension operator
\begin{equation}
\label{eq:calE}
 \mathcal{E}_{\alpha}: \Hs \rightarrow \HL(y^{\alpha},\C) 
\end{equation}
is defined as follows: If $w \in \Hs$, then $\we = \mathcal{E}_{\alpha}w \in \HL(y^{\alpha},\C) $ solves
\begin{equation}
\label{eq:mathcalW}
\DIV(y^{\alpha} \mathbf{A} \GRAD \we ) +y^{\alpha}c \we= 0 \text{ in } \C, 
\quad
\we = 0 \text{ on } \partial_L \C, 
\quad
\we = w\text{ on } \Omega \times \{0\}.
\end{equation}
References \cite{CT:10,CDDS:11} provide, for $w \in \Hs$, the estimate 
$\|\GRAD \mathcal{E}_{\alpha} w \|_{L^2(y^{\alpha},\C)} \lesssim \| w \|_{\Hs}$.
%
\end{remark}

\subsection{Solution representation}
\label{subsub:solution:_representation}

In this section we present a solution representation formula for the solution to problem \eqref{eq:weak_wave}. To accomplish this task, we first notice that the solution to problem \eqref{eq:fractional_wave} can be written as $u(x',t) = \sum_{k \in \mathbb{N}} u_k(t) \varphi_k(x')$, where, for $k \in \mathbb{N}$, the coefficient $u_k(t)$ solves 
\begin{equation} 
 \label{eq:uk}
    \partial^{2}_t u_k(t) + \lambda_k^s u_k(t) = f_k(t), \quad t > 0, \quad
     u_k(0) = g_k, \quad \partial_t u_k(0) = h_k,
\end{equation}
with $g_k = (g,\varphi_k)_{L^2(\Omega)}$, $h_k = (h,\varphi_k)_{L^2(\Omega)}$, and $f_{k}(t) = (f(\cdot,t),\varphi_k)_{L^2(\Omega)}$. We recall that the sequence $\{\lambda_k, \varphi_k \}_{k \in \mathbb{N}}$ corresponds to the eigenpairs of the operator $\calL$ and are defined by \eqref{eq:eigenpairs}. Basic computations reveal, for $k \in \mathbb{N}$, that
\begin{equation}
\label{eq:u_k_explicit}
 u_k(t) = g_k \cos\left(\lambda_k^{s/2}t\right) + h_k\lambda_k^{-\frac{s}{2}}\sin\left(\lambda_k^{s/2}t\right) + \lambda_k^{-\frac{s}{2}}\int_0^t f_k(r) \sin\left(\lambda_k^{s/2}(t-r)\right) \diff r.
\end{equation}
With these ingredients at hand, we can write the solution $\ue$ of problem \eqref{eq:weak_wave} as
\begin{equation}
\label{eq:exactforms}
  \ue(x,t) = \sum_{k \in \mathbb{N} } u_k(t) \varphi_k(x') \psi_k(y),
\end{equation}
where, for $\alpha = 1-2s$, the functions $\psi_k$ solve
\begin{equation}
\label{eq:psik}
\begin{dcases}
  \frac{\diff^2}{\diff y^2}\psi_k(y) + \frac{\alpha}{y} \frac{\diff}{\diff y} \psi_k(y) = \lambda_k \psi_k(y) , 
& y \in (0,\infty), 
\\
\psi_k(0) = 1, & \lim_{y\rightarrow \infty} \psi_k(y) = 0.
\end{dcases}
\end{equation}
If $s = \tfrac{1}{2}$, we thus have $\psi_k(y) = \exp(-\sqrt{\lambda_k}y)$ \cite[Lemma 2.10]{CT:10}; more generally, if $s \in (0,1) \setminus \{ \tfrac{1}{2}\}$, then \cite[Proposition 2.1]{CDDS:11}
\begin{equation}
\label{eq:psik_representation}
 \psi_k(y) = c_s (\sqrt{\lambda_k}y)^s K_s(\sqrt{\lambda_k}y),
\end{equation}
where $c_s = 2^{1-s}/\Gamma(s)$ and $K_s$ denotes the modified Bessel function of the second kind. We refer the reader to \cite[Chapter 9.6]{Abra} and \cite[Chapter 7.8]{MR0435697} for a comprehensive treatment of the Bessel function $K_s$. We immediately comment the following property that the function $\psi_k$ satisfies:
\[
  \lim_{s \rightarrow \frac{1}{2}}\psi_k(y)= \exp(-\sqrt{\lambda_k}y) \quad \forall y>0.
\]

In addition, for $a,b \in \mathbb{R}^{+}$, $a < b$, we have \cite[formula (2.33)]{NOS}
\begin{equation}
\label{eq:int_a_b}
 \int_{a}^b y^{\alpha}\left( \lambda_k \psi_k(y)^2 + \psi_k'(y)^2\right) \diff y = y^{\alpha} \psi_k(y) \psi_k'(y)|_{a}^b,
\end{equation}
\cite[formula (2.32)]{NOS}
\begin{equation}
 \label{eq:exp_decay_psik}
 |y^{\alpha} \psi_k(y) \psi_k'(y)| \lesssim \lambda_k^s e^{-\sqrt{\lambda_k}y}, \quad y \geq 1,
\end{equation}
and \cite[formula (2.31)]{NOS}
\begin{equation}
 \label{eq:limit_psik}
 \lim_{y \downarrow 0^+} \frac{y^{\alpha} \psi_k'(y)}{d_s \lambda_k^s} = -1.
\end{equation}

\section{Regularity}
\label{sec:Regularity}

In this section we review and derive space, time, and space--time regularity results for the solution $\ue$ of problem \eqref{eq:weak_wave}.


\subsection{Space regularity}
\label{sub:space_regularity}

To present the space regularity properties of $\ue$, we introduce the weight
\begin{equation}
\label{eq:weight}
\omega_{\beta,\theta}(y) = y^{\beta}e^{\theta y}, \qquad 0 \leq \theta < 2 \sqrt{\lambda_1},
\end{equation}
where $\beta \in \R $ will be specified later. With this weight at hand, we define the norm
\begin{equation}
 \label{eq:weighted_norm}
 \| v \|_{L^2(\omega_{\beta,\theta},\C)} 
:= \left( \int_0^{\infty} \int_{\Omega} \omega_{\beta,\theta}(y) 
          |v(x',y)|^2 \diff x' \diff y
 \right)^{\frac{1}{2}}.
\end{equation}

We now present the following pointwise, in time, bounds for $\ue$ \cite[Theorem 4.5]{OS:17}.

\begin{proposition}[pointwise bounds]
\label{thm:pointwisebounds}
Let $\ue$ solve problem \eqref{eq:wave_alpha_extension}--\eqref{eq:initial_cond} for $s\in (0,1)$. Let $0 \leq \sigma < s$ and $0 \leq \nu < 1+s $. Then, there exists $\kappa > 1$ such that the following estimates hold for all $\ell \in \mathbb{N}_0$:
\begin{align}
\label{eq:partial_y_l+1}
  \| \partial_y^{\ell+1} \ue(\cdot,t) \|_{L^2(\omega_{\alpha+2\ell-2\sigma,\theta}, \C)}^2 &\lesssim \kappa^{2(\ell+1)} (\ell+1)!^2 \| u(\cdot,t) \|_{\mathbb{H}^{\sigma+s}(\Omega)}^2, \\
\label{eq:nablaxpartial_y_l+1}
  \| \GRAD_{x'} \partial_y^{\ell+1} \ue(\cdot,t) \|_{L^2(\omega_{\alpha+2(\ell+1)-2\nu,\theta},\C)}^2 &\lesssim \kappa^{2(\ell+1)} (\ell+1)!^2 \| u(\cdot,t) \|_{\mathbb{H}^{\nu+s}(\Omega)}^2, \\
\label{eq:deltaxpartial_y_l+1}
   \|  \mathcal{L}_{x'}  \partial_y^{\ell+1} \ue(\cdot,t) \|^2_{L^2(\omega_{\alpha+2(\ell+1)-2\nu,\theta},\C)} 
   &\lesssim \kappa^{2(\ell+1)} (\ell+1 )!^2 \| u(\cdot,t)\|_{\mathbb{H}^{1 + \nu + s}(\Omega)}^2.
\end{align}
The hidden constants are independent of $\ue$, $\ell$, and the problem data.
\label{pro:pointwisebounds}
\end{proposition}

The result below shows the spatial analyticity of $\ue$ with respect to the extended variable $y \in (0,\infty)$: $\ue$ belongs to countably normed, power--exponentially weighted Bochner spaces of analytic functions with respect to $y$ taking values in spaces $\mathbb{H}^r(\Omega)$.

\begin{proposition}[space regularity]
\label{TH:regularitygamma2}
Let $\ue$ solve \eqref{eq:wave_alpha_extension}--\eqref{eq:initial_cond} for $s\in(0,1)$. Let $0 \leq \sigma < s$ and $0 \leq \nu < 1+s $. Then, there exists $\kappa > 1$ such that the following regularity estimates hold for all $\ell \in \mathbb{N}_0$: 
\begin{multline}
\| \partial_{y}^{\ell +1 } \ue \|_{L^2(0,T;L^2(\omega_{\alpha + 2 \ell -2\sigma,\theta},\C))}^2  
\lesssim  \kappa^{2(\ell+1)}(\ell+1)!^2 \big(
\| g \|^2_{\mathbb{H}^{\sigma + s}(\Omega)} + 
\| h \|^2_{\mathbb{H}^{\sigma}(\Omega)} 
\\
+ 
\| f \|_{L^2(0,T;\mathbb{H}^{\sigma}(\Omega))}^2 \big),
\label{eq:reg_in_y_gamma_1_gammaeq2} 
\end{multline} 
\begin{multline}
\| \GRAD_{x'} \partial_{y}^{\ell +1 } \ue \|_{L^2(0,T;L^2(\omega_{\alpha + 2 (\ell+1) -2\nu,\theta},\C))}^2  
\lesssim  \kappa^{2(\ell+1)}(\ell+1)!^2 \big(
\| g \|^2_{\mathbb{H}^{\nu + s}(\Omega)} + 
\| h \|^2_{\mathbb{H}^{\nu}(\Omega)} 
\\
+ 
\| f \|_{L^2(0,T;\mathbb{H}^{\nu}(\Omega))}^2 \big).
\label{eq:reg_in_y_gamma_2_gammaeq2} 
\end{multline} 
and
\begin{multline}
\| \mathcal{L}_{x'} \partial_{y}^{\ell +1 } \ue \|_{L^2(0,T;L^2(\omega_{\alpha + 2 (\ell+1) -2\nu,\theta},\C))}^2  
\lesssim  \kappa^{2(\ell+1)}(\ell+1)!^2 \big(
\| g \|^2_{\mathbb{H}^{1+\nu + s}(\Omega)} \\ + 
\| h \|^2_{\mathbb{H}^{1+\nu}(\Omega)} 
+ 
\| f \|_{L^2(0,T;\mathbb{H}^{1+\nu}(\Omega))}^2 \big).
\label{eq:reg_in_y_gamma_3_gammaeq2} 
\end{multline} 
\label{pro:L2_time_bounds}
The hidden constants are independent of $\ue$, $\ell$, and the problem data.
\end{proposition}
\begin{proof}
In view of \eqref{eq:u_k_explicit} and the continuity estimate of Lemma \ref{le:continuity}, we conclude, for $k \in \mathbb{N}$, that
 \[
  \|u_k \|^2_{L^2(0,T)} \lesssim T g_k^2 + T \lambda_k^{-s} h_k^2 + T^2\lambda_k^{-s} \| f_k\|^2_{L^2(0,T)}.
 \]
 The desired estimates \eqref{eq:reg_in_y_gamma_1_gammaeq2}--\eqref{eq:reg_in_y_gamma_3_gammaeq2} thus follow directly from \eqref{eq:partial_y_l+1}--\eqref{eq:deltaxpartial_y_l+1}.
\end{proof}

\subsection{Time regularity}
\label{sub:time_regularity}

We begin this section by defining, for $\ell \in \{1,\dots,4 \}$,
\begin{equation}
 \label{eq:Sigma}
 \Sigma_{\ell}(f,g,h):= \| g \|_{\mathbb{H}^{(\ell+1)s}(\Omega)} + \| h \|_{\mathbb{H}^{\ell s}(\Omega)} + \|f \|_{L^2(0,T;\mathbb{H}^{\ell s}(\Omega))}.
\end{equation}
In addition, and to shorten notation, we define
\begin{equation}
 \label{eq:Sigmanew}
 \Xi(f,g,h):=  \Sigma_{4}(f,g,h) + \| \partial_t^2 f \|_{L^{2}(0,T;\Hs)}.
\end{equation}

We now derive regularity estimates in time for the solution $\ue$. These estimates will be needed in the analysis of the fully discrete schemes proposed in section \ref{sec:space_time_discretization}.

\begin{theorem}[time--regularity]
Let $\ue$ be the solution to problem \eqref{eq:wave_alpha_extension}--\eqref{eq:initial_cond} for $s\in(0,1)$. The following regularity estimates in time hold:
\begin{align}
\label{eq:partial_t_l_nabla_U}
  \| \partial_t \nabla \ue \|_{L^{\infty}(0,T;L^2(y^{\alpha},\C))} & \lesssim \Sigma_{1}(f,g,h),
  \\
  \label{eq:partial_t_l_nabla_U2}
    \| \partial_t^{2} \nabla \ue \|_{L^{\infty}(0,T;L^2(y^{\alpha},\C))} & \lesssim \Sigma_{2}(f,g,h),
\\
\label{eq:partial_t_l_nabla_U3}
  \| \partial_t^{3} \nabla \ue \|_{L^{\infty}(0,T;L^2(y^{\alpha},\C))} & \lesssim \Sigma_{3}(f,g,h) + \| \partial_t f \|_{L^2(0,T;\Hs)},
\\
\label{eq:partial_t_l_nabla_U4}
  \| \partial_t^{4} \nabla \ue \|_{L^{\infty}(0,T;L^2(y^{\alpha},\C))} & \lesssim \Xi(f,g,h).
\end{align}
In all these inequalities the hidden constants do not depend either on $\ue$ or the problem data.
\label{thm:time_regularity}
\end{theorem}

\begin{proof}
Since $\{ \varphi_k \}_{k \in \mathbb{N}}$ is an orthonormal basis of $L^2(\Omega)$ and an orthogonal basis of $(H_0^1(\Omega),a_{\Omega}(\cdot,\cdot))$, the definition of the energy norm $\normC{\cdot}$, given in \eqref{eq:norm-C}, and the properties \eqref{eq:int_a_b} and \eqref{eq:limit_psik} allow us to conclude, for $\ell \in \mathbb{N}_0$, that
\begin{align*}
\normC{ \partial_t^{\ell} \nabla \ue (\cdot,t)} ^2 & = d_s^{-1} \int_{\C} y^{\alpha} \left[ \mathbf{A} \GRAD \partial_t^{\ell}\ue(\cdot,t) \cdot \GRAD \partial_t^{\ell}\ue(\cdot,t)  + c (\partial_t^{\ell}\ue(\cdot,t))^2\right]\diff x
\\
& = d_s^{-1}  \sum_{k \in \mathbb{N}} (\partial_t^{\ell} u_k(t))^2 \int_{0}^{\infty} \left[ \lambda_k \psi_k(y)^2 + \psi_k'(y)^2 \right] \diff y  =  \sum_{k \in \mathbb{N}} (\partial_t^{\ell} u_k(t))^2 \lambda_k^s.
\end{align*}
We have thus arrived at the estimate $\normC{ \partial_t^{\ell} \nabla \ue (\cdot,t)}^2 = \| \partial_t^{\ell} u(\cdot,t) \|^2_{\Hs}$.

We now invoke the explicit representation of the coefficient $u_k(t)$, with $k \in \mathbb{N}$, which is provided in \eqref{eq:u_k_explicit}, to obtain that 
\[
\partial_t u_k(t) = -g_k \lambda_k^{s/2} \sin \left(\lambda_k^{s/2} t\right) + h_k \cos \left(\lambda_k^{s/2} t\right) + \int_0^t f_k(r) \cos \left(\lambda_k^{s/2} (t-r)\right) \diff r.
\]
This, on the basis of the definition of the norm $ \| \cdot \|_{\mathbb{H}^r(\Omega)}$, given in \eqref{def:Hs}, and an application of Lemma \ref{le:continuity}, reveal that
\[
\normC{ \partial_t \nabla \ue (\cdot,t)}^2\lesssim
  \sum_{k \in \mathbb{N}} (\partial_t u_k(t))^2 \lambda_k^s \lesssim \| g\|^2_{\mathbb{H}^{2s}(\Omega)} + \| h \|^2_{\mathbb{H}^{s}(\Omega)} + \sum_{k \in \mathbb{N}} \lambda_k^s \| f_k\|^2_{L^2(0,T)},
\]
which implies the desired estimate \eqref{eq:partial_t_l_nabla_U}. 

To derive \eqref{eq:partial_t_l_nabla_U2} we invoke, again, the representation formula \eqref{eq:u_k_explicit} and write
\begin{multline}
\partial_t^2 u_k(t) = -g_k \lambda_k^{s} \cos \left(\lambda_k^{s/2} t\right) - h_k \lambda_k^{s/2} \sin \left(\lambda_k^{s/2} t\right)
\\ + f_k(t)  - \lambda_k^{s/2}  \int_0^t f_k(r) \sin \left(\lambda_k^{s/2} (t-r)\right) \diff r.
\label{eq:second_derivatives}
\end{multline}
We thus use the definition of the norm $ \| \cdot \|_{\mathbb{H}^r(\Omega)}$ 
to arrive at \eqref{eq:partial_t_l_nabla_U2}.

The estimates \eqref{eq:partial_t_l_nabla_U2}--\eqref{eq:partial_t_l_nabla_U4} follow similar arguments upon taking derivatives to the explicit representation of the coefficient $u_k(t)$, with $k \in \mathbb{N}$, provided in \eqref{eq:u_k_explicit}. This concludes the proof.
\end{proof}
\subsection{Space--time regularity}
\label{sub:space_time_regularity}

We present the following regularity result in space and time.

\begin{theorem}[space--time regularity]\label{thm:st_reg}
Let $\ue$ solve \eqref{eq:wave_alpha_extension}--\eqref{eq:initial_cond} for $s\in(0,1)$. Let $0 \leq \sigma < s$ and $0 \leq \nu < 1+s $. Then, there exists $\kappa > 1$ such that the following regularity estimates hold for all $\ell \in \mathbb{N}_0$:
\begin{multline}
\| \partial_t^2 \partial_{y}^{\ell +1 } \ue \|_{L^2(0,T;L^2(\omega_{\alpha + 2 \ell -2\sigma,\theta},\C))}^2  
\lesssim  \kappa^{2(\ell+1)}(\ell+1)!^2 \big(
\| g \|^2_{\mathbb{H}^{\sigma + 3s}(\Omega)} 
\\
+ 
\| h \|^2_{\mathbb{H}^{\sigma + 2s}(\Omega)} 
+ \| f \|_{L^2(0,T;\mathbb{H}^{\sigma+2s}(\Omega))}^2\big),
\label{eq:reg_in_y_gamma_1_gammaeq2_partial_t_2} 
\end{multline}
\begin{multline}
\| \partial_t^2\GRAD_{x'} \partial_{y}^{\ell +1 } \ue \|_{L^2(0,T;L^2(\omega_{\alpha + 2 (\ell+1) -2\nu,\theta},\C))}^2  
\lesssim  \kappa^{2(\ell+1)}(\ell+1)!^2 \big(
\| g \|^2_{\mathbb{H}^{\nu + 3s}(\Omega)} 
\\
+ 
\| h \|^2_{\mathbb{H}^{\nu + 2s}(\Omega)} 
+ \| f \|_{L^2(0,T;\mathbb{H}^{\nu+2s}(\Omega))}^2\big),
\label{eq:reg_in_y_gamma_2_gammaeq2_partial_t_2} 
\end{multline}
\begin{multline}
\| \partial_t^2\mathcal{L}_{x'} \partial_{y}^{\ell +1 } \ue \|_{L^2(0,T;L^2(\omega_{\alpha + 2 (\ell+1) -2\nu,\theta},\C))}^2  
\lesssim  \kappa^{2(\ell+1)}(\ell+1)!^2 \big(
\| g \|^2_{\mathbb{H}^{1+\nu + 3s}(\Omega)} 
\\
+ 
\| h \|^2_{\mathbb{H}^{1+\nu + 2s}(\Omega)} 
+ \| f \|_{L^2(0,T;\mathbb{H}^{1+\nu+2s}(\Omega))}^2\big).
\label{eq:reg_in_y_gamma_3_gammaeq2_partial_t_2} 
\end{multline}
The hidden constants do not depend either on $\ue$ or the problem data.
\label{thm:time_space_regularity}
\end{theorem}
\begin{proof}
Similar arguments to the ones used to derive \eqref{eq:partial_y_l+1} reveal that
\[
  \| \partial_t^2 \partial_y^{\ell+1} \ue(\cdot,t) \|_{L^2(\omega_{\alpha+2\ell-2\sigma,\theta}, \C)}^2 \lesssim \kappa^{2(\ell+1)} (\ell+1)!^2 \| \partial_t^2 u(\cdot,t) \|_{\mathbb{H}^{\sigma+s}(\Omega)}^2.
\]
To control the right--hand side of the previous inequality we invoke formula \eqref{eq:second_derivatives} and the definition of the $\mathbb{H}^r(\Omega)$--norm, given in \eqref{def:Hs}. These arguments reveal the estimate
\begin{equation}\label{eq:utt_sigma_s}
  \begin{split}    
 \| \partial_t^2 u(\cdot,t) \|_{\mathbb{H}^{\sigma+s}(\Omega)}^2 \lesssim &\ \| g \|^2_{\mathbb{H}^{\sigma + 3s}(\Omega)} + \| h \|^2_{\mathbb{H}^{\sigma + 2s}(\Omega)} 
\\&+ \sum_{k \in \mathbb{N}} \lambda_k^{\sigma + s}  \left( f_k^2(t) +  \lambda_k^s \| f_k \|^2_{L^2(0,T)} \right).
  \end{split}
\end{equation}
This yields \eqref{eq:reg_in_y_gamma_1_gammaeq2_partial_t_2}. Similar arguments allow us to derive the regularity estimates \eqref{eq:reg_in_y_gamma_2_gammaeq2_partial_t_2} and \eqref{eq:reg_in_y_gamma_3_gammaeq2_partial_t_2}.
\end{proof}

As it will be used in the analysis that follows, we introduce
\begin{equation}
 \mathfrak{A}(f,g,h) =  \| g \|_{\mathbb{H}^{1+3s}(\Omega)} 
 + \| h \|_{\mathbb{H}^{1+2s}(\Omega)}
 + \| f \|_{L^2(0,T;\mathbb{H}^{1+2s}(\Omega))} ,
 \label{eq:frakA}
\end{equation}
and notice that, if $0 \leq \sigma < s$, then
\begin{multline}
\label{eq:estimate_frakA}
\| \partial_t^2\partial_{y}^{\ell +1 } \ue \|_{L^2(0,T;L^2(\omega_{\alpha + 2 \ell-2\sigma,\theta},\C))}^2 + \| \partial_t^2\GRAD_{x'} \partial_{y}^{\ell +1 } \ue \|_{L^2(0,T;L^2(\omega_{\alpha + 2 (\ell+1),\theta},\C))}^2 
\\
+ \|\partial_t^2 \mathcal{L}_{x'} \partial_{y}^{\ell +1 } \ue \|_{L^2(0,T;L^2(\omega_{\alpha + 2 (\ell+1),\theta},\C))}^2 \lesssim  \kappa^{2(\ell+1)}(\ell+1)!^2 \mathfrak{A}(f,g,h)^2.
\end{multline}

As a consequence of the arguments elaborated in the proof of the previous theorem (see \eqref{eq:utt_sigma_s}) and the previous definition, we can thus immediately arrive at the following regularity estimate.
\begin{corollary}[space--time regularity]
\label{cor:st_reg_u}
Let $u$ solve \eqref{eq:fractional_wave} for $s \in (0,1)$. Then
\begin{equation}
\| \partial_t^2 u \|_{\mathbb{H}^{1+s}(\Omega)} \lesssim \mathfrak{A}(f,g,h).
\end{equation}
The hidden constant does not depend either on $u$ or the problem data.
\end{corollary}
\section{Truncation}
\label{sec:truncation}

A first step towards space--discretization is to truncate the semi--infinite cylinder $\C$. In the next result we show that the solution $\ue$ to problem \eqref{eq:wave_alpha_extension}--\eqref{eq:initial_cond} decays exponentially in the extended variable $y$ for a.e. $t \in (0,T)$. This suggests to truncate $\C$ to $\C_{\Y} = \Omega \times (0,\Y)$, with a suitable truncation parameter $\Y$, and seek solutions in this bounded domain.

\begin{proposition}[exponential decay]
Let $s \in (0,1)$ and $\ue$ be the solution to \eqref{eq:weak_wave}. Then, for every $\Y \geq 1$, we have that 
\begin{equation}
\| \nabla \ue \|_{L^2(0,T;L^2(y^{\alpha},\Omega \times (\Y,\infty)))}  
\lesssim  e^{-\sqrt{\lambda_1}  \Y /2 } \Lambda(f,g,h),
\label{eq:exp_decayment}
\end{equation}
where $\lambda_1$ denotes the first eigenvalue of $\calL$ and $\Lambda(f,g,h)$ is defined in \eqref{eq:Lambda}. 
\label{pro:exp_decayment}
\end{proposition}
\begin{proof}
We invoke \eqref{eq:exactforms} and the fact that $\{ \varphi_k \}_{k \in \mathbb{N}}$ is an orthonormal basis of $L^2(\Omega)$ and an orthogonal basis of $(H_0^1(\Omega),a_{\Omega}(\cdot,\cdot))$ to conclude that
\[
 \int_0^T \int_{\C \setminus \C_{\Y}} y^{\alpha} \left( \mathbf{A} \GRAD \ue \cdot \GRAD \ue  + c \ue^2\right)\diff x \diff t = \int_0^T \sum_{k \in \mathbb{N}} u_k^2(t) \int_{\Y}^{\infty} \left( \lambda_k \psi_k^2 + \psi_k'(y)^2\right) \diff y  \diff t.
\]
We now apply formulas \eqref{eq:int_a_b} and \eqref{eq:exp_decay_psik} to obtain that
\begin{multline*}
  \int_0^T \int_{\C \setminus \C_{\Y}} y^{\alpha} \left( \mathbf{A} \GRAD \ue \cdot \GRAD \ue  + c \ue^2\right)\diff x \diff t  = \sum_{k \in \mathbb{N}}  | \Y^{\alpha} \psi_k(\Y) \psi_k'(\Y) | \int_0^T  u_k^2(t) \diff t 
  \\
   \lesssim \sum_{k \in \mathbb{N}}  e^{-\sqrt{\lambda_k}  \Y } \lambda_k^s \| u_k\|_{L^2(0,T)}^2 \lesssim e^{-\sqrt{\lambda_1}  \Y }\| u \|^2_{L^2(0,T;\Hs)}.
\end{multline*}
The desired estimate \eqref{eq:exp_decayment} is thus a consequence of the energy estimate \eqref{eq:energy_estimate_2} for the solution $u$ to problem \eqref{eq:fractional_wave}.
\end{proof}

To describe the truncated version of \eqref{eq:weak_wave}, we define the weighted Sobolev space 
\begin{equation*}
  \HL(y^{\alpha},\C_\Y)  = \left\{ w \in H^1(y^{\alpha},\C_\Y): w = 0 \text{ on }
    \partial_L \C_\Y \cup \Omega_{\Y} \right\},\\
\end{equation*}
and the bilinear form $a_\Y: \HL(y^{\alpha},\C_{\Y}) \times \HL(y^{\alpha},\C_{\Y})$ as
\begin{equation}
\label{eq:a_Y}
  a_\Y(w,\phi) = \frac{1}{d_s} \int_{\C_\Y} y^{\alpha} \left( \mathbf{A}(x) \GRAD w \cdot \GRAD \phi +  c(x') w \phi \right) \diff x,
\end{equation}
where $\C_{\Y} = \Omega \times (0,\Y)$ and $\Omega_{\Y} = \Omega \times \{ \Y \}$.

On the basis of the results of Proposition \ref{pro:exp_decayment}, we thus consider the following truncated problem: Find $\Ucal \in L^{\infty}(0,T; \HL(y^{\alpha},\C_{\Y}))$ with $\tr \partial_t \Ucal \in L^{\infty}(0,T;L^2(\Omega))$ and $\tr \partial_{t}^{2} \Ucal \in L^2(0,T;\Hsd)$ such that $\tr \Ucal(0) = g$, $\tr \partial_t \Ucal(0) = h$ and, for a.e. $t \in (0,T)$,
\begin{equation}
\label{eq:truncated_weak_wave}
\langle \tr \partial_t^{2} \Ucal, \tr \phi \rangle + 
a_\Y(\Ucal,\phi) = \langle f, \tr \phi \rangle \qquad \forall \phi \in \HL(y^{\alpha},\C_{\Y}).
\end{equation}

We define $\calH_{\alpha} :\Hs \rightarrow \HL(y^{\alpha},\C_{\Y})$, the truncated $\alpha$-harmonic extension operator, as follows: if $w \in \Hs$, then $\calW = \calH_\alpha w$ solves
\begin{equation}
\label{eq:calH}
\DIV(y^{\alpha} \mathbf{A} \nabla \mathcal{W}) +y^{\alpha}c\mathcal{W}  = 0 \text{ in } \C_\Y, 
\quad
\mathcal{W} = 0 \text{ on } \partial_L \C_\Y \cup \Omega_{\Y}, 
\quad
\mathcal{W} = w\text{ on } \Omega \times \{0\}.
\end{equation}

\begin{remark}[initial data]
\label{rem:initial_data_truncated}
\rm
As in Remark \ref{rem:initial_data}, we define $\Ucal(0) = \mathcal{H}_{\alpha}g$ and $\Ucal_t(0) = \mathcal{H}_{\alpha}h$, where $\mathcal{H}_{\alpha}$ is defined by \eqref{eq:calH}. References \cite{CT:10,CDDS:11} provide the estimates 
$\|\Ucal(0) \|_{L^2(y^{\alpha},\C)} \lesssim \| g \|_{\Hs}$ and $\|\partial_t \Ucal(0) \|_{L^2(y^{\alpha},\C)} \lesssim \| h \|_{\Hs}$.
%
\end{remark}

The following result shows that by considering \eqref{eq:truncated_weak_wave} instead of \eqref{eq:weak_wave} we only incur an exponentially small error

\begin{lemma}[exponential error estimate]
Let $\ue$ and $\Ucal$ be the solutions of problems \eqref{eq:weak_wave} and \eqref{eq:truncated_weak_wave}, respectively.
Then, for every $s \in (0,1)$ and $\Y \geq 1$, we have 
\begin{multline}
\label{eq:exp_convergence}
\| \tr \partial_t (\ue-\Ucal)\|^2_{L^{\infty}(0,T;L^2(\Omega))}  +  \| \GRAD(\ue-\Ucal)\|^2_{L^{\infty}(0,T;L^2(y^\alpha, \C_{\Y}) )} 
\\
\lesssim e^{-\sqrt{\lambda_1} \Y/2} \Sigma_1(f,g,h)
\end{multline}
where $\Sigma_1$ is defined by \eqref{eq:Sigma} and the hidden constant does not depend on either $\ue$, $\Ucal$, or the problem data.
\label{lemma:exp_estimate}
\end{lemma}
\begin{proof}
We begin the proof by defining the cutoff function $\rho \in W^{1,\infty}(0,\infty)$ as
\[
  \rho(y) = 1 \quad 0 \leq y \leq \frac\Y2, \qquad
  \rho(y) = \frac2\Y \left( \Y - y \right) \quad  \frac\Y2 < y < \Y, \qquad
  \rho(y) =0 \quad \Y \geq y.
\]

Notice that by a trivial zero extension we realize that $\Ucal$ belongs to $\HL(y^{\alpha},\C)$. We are thus allow to set $\phi= \partial_t(\Ucal -\rho \ue)$ in problems \eqref{eq:weak_wave} and \eqref{eq:truncated_weak_wave}. With these choices of test functions, we subtract the ensuing equalities and obtain that 
\[
\langle \tr \partial_t^{2} (\Ucal-\ue), \tr \partial_t (\Ucal-\rho \ue) \rangle + a_\Y(\Ucal-\ue, \partial_t(\Ucal-\rho \ue)) = 0.
\]
This expression yields
\[
 \frac{1}{2} \partial_t \| \tr \partial_t (\Ucal-\ue) \|_{L^2(\Omega)}^2 + \frac{1}{2} \partial_t a_{\Y}(\Ucal-\ue,\Ucal-\ue) =  a_{\Y}(\Ucal-\ue,\partial_t(\rho \ue - \ue))
\]
We thus integrate over time and use that $\tr \partial_t (\Ucal-\ue)|_{t=0} = 0$ to arrive at
\begin{multline}
 \| \tr \partial_t (\Ucal-\ue)(t) \|_{L^2(\Omega)}^2 +   \| \GRAD( \Ucal-\ue )(t) \|_{L^2(y^{\alpha},\C_{\Y})}^2
 \\
 \lesssim \| \GRAD(\Ucal(0) - \ue(0)) \|^2_{L^2(y^{\alpha},\C_{\Y})} + \int_0^t |a_{\Y}(\Ucal-\ue,\partial_{\zeta}(\rho \ue - \ue))|\diff \zeta =: \mathrm{I} + \mathrm{II}.
 \label{eq:first_truncation_estimate}
\end{multline}

It thus remains to bound the right--hand side of \eqref{eq:first_truncation_estimate}. First, in view of the fact that $\ue(0) = \calE_{\alpha} g$ and $\Ucal(0) = \calH_{\alpha} g$, with $\calE_{\alpha}$ and $\calH_{\alpha}$ being defined as in \eqref{eq:calE} and \eqref{eq:calH}, respectively, the results of \cite[Lemma 3.3]{NOS} allow us to conclude that
\begin{equation}
 \mathrm{I} = \| \GRAD( \mathcal{H}_{\alpha} - \mathcal{E}_{\alpha})g \|_{L^2(y^{\alpha},\C_{\Y})} \lesssim e^{-\sqrt{\lambda}_1 \Y/4} \| g \|_{\Hs}.
  \label{eq:second_truncation_estimate}
\end{equation}
To bound the term $\mathrm{II}$, we notice that if $y \leq \Y/2$, $(\rho-1)\ue \equiv 0$. If $y > \Y/2 $, then
\[
|\GRAD (\rho -1 )\partial_t \ue|^2 \leq 2 \left( \frac{4}{\Y^2} |\partial_t \ue|^2 + |\partial_t \GRAD \ue|^2\right).
\]
Consequently,
\[
  \| \GRAD (\rho-1) \partial_t \ue \|_{L^2(y^\alpha,\C_\Y)}^2 
  \lesssim \frac{1}{\Y^2}\int_{\frac\Y2}^\Y \int_{\Omega} y^{\alpha}|\partial_t \ue|^2\diff x' \diff y +
  \int_{\frac\Y2}^\Y \int_{\Omega}y^{\alpha} |\partial_t \nabla \ue|^2\diff x' \diff y. 
\]
A weighted Poincar\'e inequality, an application of \eqref{eq:int_a_b} and \eqref{eq:exp_decay_psik}, as in the proof of Proposition \ref{pro:exp_decayment}, and the use of the estimate \eqref{eq:exp_decayment} allow us to conclude that
\begin{align*}
\| \GRAD(\rho-1)\partial_t\ue \|^2_{L^2(y^{\alpha},\C_{\Y})} 
& \lesssim 
\| \partial_t \GRAD \ue \|^2_{L^2(y^{\alpha},\Omega \times (\Y/2,\Y))}
\leq \| \partial_t \GRAD \ue \|^2_{L^2(y^{\alpha},\Omega \times (\Y/2,\infty))}
\\
& \lesssim e^{-\sqrt{\lambda_1} \Y /2} \| \partial_t \GRAD \ue \|^2_{L^2(y^{\alpha},\C)}.
\end{align*}
The regularity estimate \eqref{eq:partial_t_l_nabla_U} thus implies that
\[
 \textrm{II} \leq C e^{-\sqrt{\lambda}_1 \Y/2}  \Sigma^2_{1}(f,g,h) + \tfrac{1}{2}\| \GRAD( \Ucal-\ue) \|^2_{L^{\infty}(0,T;L^2(y^{\alpha},\C_{\Y}))},
\]
where $C$ denotes a positive constant. Replacing the previous estimate for $\textrm{II}$ and the one in \eqref{eq:second_truncation_estimate} for $\textrm{I}$ into \eqref{eq:first_truncation_estimate} we obtain the desired exponential error estimate \eqref{eq:exp_convergence}.
\end{proof}

\section{Space and time discretization}
\label{sec:space_time_discretization}

In this section we present two fully discrete schemes for approximating the solution to problem \eqref{eq:fractional_wave}. In view of the localization results of Theorem \ref{thm:CS_NOT_ST} and the exponential error estimate \eqref{eq:exp_convergence} we shall thus discretize the truncated problem \eqref{eq:truncated_weak_wave}. We begin by setting notation on finite element spaces and introducing a finite element approximation in $\Omega$.

\subsection{Finite element methods}
\label{subsec:fems}
We follow \cite{BMNOSS:17} and present a scheme based on the tensorization of a first--degree FEM in $\Omega$ with a suitable $hp$--FEM in the extended variable. The scheme achieves log--linear complexity with respect to the number of degrees of freedom in $\Omega$.
To describe it, on the interval $[0,\Y]$, we consider geometric meshes $\calG^M_{\sigma} = \{ I_m \, | \,m =1,\dots M \}$ with $M$ elements and grading factor $\sigma \in (0,1)$: 
\begin{equation}
\label{eq:graded_meshes}
 I_1 = [0,\Y \sigma^{M-1}], \qquad I_i = [\Y \sigma^{M-m+1},\Y \sigma^{M-m}], \quad m \in \{ 2,\ldots,M \}.
\end{equation}
Notice that the meshes $\calG^M_{\sigma}$ are refined towards $y=0$ in order to capture the singular behavior exhibited by the solution $\ue$ on the extended variable $y$ as described in Propositions \ref{pro:pointwisebounds} and \ref{pro:L2_time_bounds}. On the aforementioned meshes, we consider a \emph{linear degree vector} $\bmr =(r_1,\dots,r_M) \in \mathbb{N}^M$ with slope $\slope$: 
$r_m := \max\{1,\lceil \slope m \rceil \}$, where $m=1,2,...,M$. We define the following finite element space:
\[
S^\bmr((0,\Y),\calG_{\sigma}^M) 
= 
\left \{ v_M \in C[0,\Y]: v_M|_{I_m} \in \mathbb{P}_{r_m}(I_m), I_m \in \calG_{\sigma}^M, m =1, \dots, M \right \},
\]
and the subspace of $S^\bmr((0,\Y),\calG^M)$ containing functions that vanish at $y=\Y$:
\[
S_{\{\Y\}}^\bmr((0,\Y),\calG_{\sigma}^M) = \left \{ v_M \in S^\bmr((0,\Y),\calG_{\sigma}^M): v_M(\Y) = 0 \right \}.
\]

Let $\T = \{ K\}$ be a conforming partition of $\bar \Omega$ into simplices $K$. We denote by $\Tr$ a collection of conforming and shape regular meshes that are refinements of an original mesh $\T_0$. For $\T \in \Tr$, we define $h_{\T} = \max \{\diam(K) : K \in \T \}$ and $N = \# \T$, the number of degrees of freedom of $\T$. We introduce 
%
the finite element space:
\[
S^1_0(\Omega,\T)
= 
\left \{ v_h \in C(\bar \Omega): v_h|_{K} \in \mathbb{P}_{1}(K) \quad 
\forall K \in \T, \ v_h|_{\partial \Omega} = 0 \right \}.
\]

With the meshes $\calG_{\sigma}^M$ and $\T$ at hand, we define $\T_{\Y} = \T \otimes \calG_{\sigma}^M$ and the finite--dimensional \emph{tensor product space}
\begin{equation}\label{eq:TPFE}
\V^{1,\bmr}_{N,M}(\T_{\Y}) := S^1_0(\Omega,\T) \otimes S_{\{ \Y\}}^\bmr((0,\Y),\calG_{\sigma}^M)
\subset \HL(y^{\alpha},\C_{\Y}).
\end{equation}
We write $\V(\T_{\Y})$ if the arguments are clear from the context.

Finally, we recall the standard $L^2(\Omega)$--orthogonal projection operator $\Pi_{x'}$:
\begin{equation}
 \label{eq:ortogonal_projection}
 \Pi_{x'} : L^2(\Omega) \rightarrow S^1_0(\Omega,\T), \quad (\Pi_{x'}w,W)_{L^2(\Omega)} = (w,W)_{L^2(\Omega)} \quad \forall W \in S^1_0(\Omega,\T).
\end{equation}
If $\T$ is quasi--uniform, then \cite[Lemma 1.131]{Guermond-Ern} 
\begin{equation}
\label{eq:ortogonal_projection_bounded}
\| \Pi_{x'}w \|_{L^2(\Omega)} \leq  \| w \|_{L^2(\Omega)},
\quad 
\| \nabla \Pi_{x'}w \|_{L^2(\Omega)} \lesssim  \| \nabla w \|_{L^2(\Omega)} 
\quad \forall w \in H^1(\Omega).
\end{equation}
If, in addition, $w \in H^2(\Omega)$, then \cite[Proposition 1.134]{Guermond-Ern} 
\begin{equation}
 \label{eq:ortogonal_projection_approximation}
\|w-\Pi_{x'}w \|_{L^2(\Omega)} +  h_{\T} \| \nabla (w-\Pi_{x'}w) \|_{L^2(\Omega)}  \lesssim  h_{\T}^2 | v |_{H^2(\Omega)}.
\end{equation}

\subsection{Weighted elliptic projector}
We define the \emph{weighted elliptic projector}
\[
  G_{\T_{\Y}}: \HL(y^{\alpha},\C_\Y) \rightarrow \V(\T_\Y)
\]
such that, for $w \in \HL(y^\alpha,\C_\Y)$, it is given by
\begin{equation}
 \label{eq:elliptic_projection}
a_\Y \left( G_{\T_{\Y}}w , W \right)  = a_\Y(w, W) \quad \forall W \in \V(\T_{\Y}).
\end{equation}
This operator is stable in $\HL(y^{\alpha},\C_\Y)$ \cite[Proposition 26]{NOS3}:
\begin{equation}
 \label{eq:G_stable}
\| \nabla G_{\T_{\Y}} w \|_{L^2(y^{\alpha},\C_\Y)} \lesssim \| \nabla w\|_{L^2(y^{\alpha},\C_\Y)} \qquad
\forall w \in \HL(y^{\alpha},\C_\Y).
\end{equation}

In what follows we present approximation properties for $G_{\T_{\Y}}$.

\begin{lemma}[error estimates for $G_{\T_{\Y}}$]\label{lemma:GT_est}
Let $\calG_{\sigma}^M$ be the geometric mesh defined in \eqref{eq:graded_meshes} where $\Y \sim | \log h_{\T}|$ with a sufficiently large constant. Let $w \in \mathbb{H}^s(\Omega)$. If $\calW$ denotes the truncated $\alpha$--harmonic extension of $w$, then there exists a minimal slope $\slope_{min}$ such that for linear degree vectors $\bmr$ with slope $\slope \ge \slope_{min}$ there holds 
\begin{equation}
 \label{eq:G_approx_1}
 \| \nabla (\calW - G_{\T_{\Y}} \calW) \|_{L^2(y^{\alpha},\C_\Y)} \lesssim h_{\T} \| w \|_{\mathbb{H}^{1+s}(\Omega)}.
\end{equation}
In addition, if $\we$ denotes the $\alpha$--harmonic extension of $w$, \ie the solution to \eqref{eq:mathcalW} with $w \in \mathbb{H}^s(\Omega)$ as a datum, then
\begin{equation}
 \label{eq:G_approx_2}
\| \tr (\we - G_{\T_{\Y}} \we) \|_{\Hs} \lesssim \| \nabla (\we - G_{\T_{\Y}} \we) \|_{L^2(y^{\alpha},\C)} \lesssim h_{\T} \| w \|_{\mathbb{H}^{1+s}(\Omega)}.
\end{equation}
In both inequalities the hidden constants are independent of $\calW$, $\we$, $w$ and $h_{\T}$.
\label{lemma:elliptic_projection}
\end{lemma}
\begin{proof}
Let $\Pi_{y, \{ \Y \}}^ {\bmr }$ and $\Pi_{x'}$ be the univariate $hp$--interpolation operator of \cite[Section 5.5.1]{BMNOSS:17} and the standard $L^2(\Omega)$--projection operator defined in \eqref{eq:ortogonal_projection}, respectively:
\begin{equation}
\Pi_{y,\{ \Y \}}^{\bmr}: C([0,\Y]) \rightarrow S_{\{\Y\}}^\bmr((0,\Y),\calG_{\sigma}^M),
\quad
\Pi_{x'}: L^2(\Omega) \rightarrow S^1_0(\Omega,\T).
\end{equation}

Set $W = \Pi_{x'} \otimes \Pi_{y, \{ \Y \}}^ {\bmr } \mathcal{W}$. Since $W \in \V(\T_{\Y})$, Galerkin orthogonality and definition \eqref{eq:elliptic_projection} yield
\[
 \| \nabla (\calW - G_{\T_{\Y}} \calW) \|^2_{L^2(y^{\alpha},\C_\Y)} \lesssim a_{\Y}(\calW - G_{\T_{\Y}} \calW,\calW - W).
\]
It suffices to bound $ \| \nabla (\calW - W) \|_{L^2(y^{\alpha},\C_\Y)}$. The stability properties of $\Pi_{x'}$, as described in \eqref{eq:ortogonal_projection_bounded}, reveal that
\[
 \| \nabla (\calW - W) \|_{L^2(y^{\alpha},\C_\Y)} \lesssim \| \nabla(\calW - \Pi_{x'} \calW)\|_{L^2(y^{\alpha},\C_{\Y})} +  \| \nabla(\calW -  \Pi_{y, \{ \Y \}}^ {\bmr } \mathcal{W})\|_{L^2(y^{\alpha},\C_{\Y})}.
\]
The estimate  \eqref{eq:G_approx_1} thus follows from the approximation properties of $\Pi_{x'}$ as described in \eqref{eq:ortogonal_projection_approximation}, the exponential interpolation error estimates of \cite[Lemma 5.13]{BMNOSS:17}, and the regularity properties of $\mathcal{W}$ \cite[Theorem 4.7]{BMNOSS:17}. The estimate \eqref{eq:G_approx_2} follows similar arguments by using first the exponential decay of $\we$ in the extended dimension \cite[Proposition 3.1]{NOS}:
\begin{align*}
 \| \nabla (\we - G_{\T_{\Y}} \we) \|_{L^2(y^{\alpha},\C)} & \leq \| \nabla \we  \|_{L^2(y^{\alpha},\C \setminus \C_{\Y})}  +  \| \nabla (\we - G_{\T_{\Y}} \we) \|_{L^2(y^{\alpha},\C_{\Y})} 
 \\
 & \lesssim  e^{-\sqrt{\lambda_1} \Y /2 }+  \| \nabla (\we - G_{\T_{\Y}} \we) \|_{L^2(y^{\alpha},\C_{\Y})}.
\end{align*}
This concludes the proof.
\end{proof}

The following improved estimate for the weighted elliptic projection $G_{\T_{\Y}}$ in the $L^2(\Omega)$--norm can be obtained by invoking the estimates of Lemma \ref{lemma:elliptic_projection} and the arguments elaborated in the proof of \cite[Proposition 28]{NOS3}.

\begin{lemma}[$L^2(\Omega)$--error estimates for $G_{\T_{\Y}}$]
Let $\calG_{\sigma}^M$ be the geometric mesh defined in \eqref{eq:graded_meshes} where $\Y \sim | \log h_{\T}|$ with a sufficiently large constant. Let $w \in \mathbb{H}^{1+s}(\Omega)$. If $\calW$ denotes the truncated $\alpha$--harmonic extension of $w$, then there exists a minimal slope $\slope_{min}$ such that for linear degree vectors $\bmr$ with slope $\slope \ge \slope_{min}$ there holds 
\begin{equation}
 \label{eq:G_approx_3}
 \| \tr( \calW - G_{\T_{\Y}} \calW)  \|_{L^2(\Omega)} \lesssim h_{\T}^{1+s} \| w \|_{\mathbb{H}^{1+s}(\Omega)}.
\end{equation}
In addition, if $\we$ denotes the $\alpha$--harmonic extension of $w$, \ie the solution to \eqref{eq:mathcalW} with $w \in \mathbb{H}^{1+s}(\Omega)$ as a datum, then
\begin{equation}
 \label{eq:G_approx_4}
\| \tr (\we - G_{\T_{\Y}} \we) \|_{L^2(\Omega)} \lesssim h_{\T}^{1+s} \| w \|_{\mathbb{H}^{1+s}(\Omega)}.
\end{equation}
In both inequalities the hidden constants are independent of $\calW$, $\we$, $w$ and $h_{\T}$.
\label{lemma:elliptic_projection_2}
\end{lemma}

\subsection{Time discretization}


Let $K \in \mathbb{N}$ be the number of time steps. We define the uniform time step as $\dt=T/K$, and we set $t_k = k \dt$, $k=0,\dots, K$. If $\Xcal$ is a normed space with norm $\| \cdot \|_{\Xcal}$, then for $w \in C([0,T];\Xcal)$ we denote $w_k = w(t_k) \in \Xcal$ and $w_{\dt} = \{ w_k \}_{k=0}^{K} \subset \Xcal$. In addition, for $w_{\dt} \subset \Xcal$ and $p \in [1,\infty)$, we define
\begin{equation}
\label{eq:definition_norm}
\| w_{\dt} \|_{\ell^p(\Xcal)} = \left( \sum_{k=1}^K \dt \| w_k\|_{\Xcal}^p \right)^{\frac{1}{p}}, 
\quad
\| w_{\dt} \|_{\ell^{\infty}(\Xcal)} = \max_{0 \leq k \leq K} \| w_k \|_{\Xcal}.
\end{equation}
For a sequence of time--discrete functions $w_{\dt} \subset \Xcal$, we define, for $k=0,\dots,K-1$,
\begin{equation}
\label{eq:frakd}
 \mathfrak{d}w_{k+1}:= (\dt)^{-1}(w_{k+1} - w_k), \qquad w_{k+1/2}:= \tfrac12(w_{k+1} + w_k).
\end{equation}
We also define, for $k = 1, \dots, K-1$, 
\begin{equation}
\label{eq:barwn}
\mathfrak{c}w_k := \tfrac{1}{4}(w_{k+1}+2w_k+w_{k-1})  = \tfrac12(w_{k+1/2}+w_{k-1/2}),
\end{equation}
and
\begin{equation}
\label{eq:partial_discrete_2}
\mathfrak{d}^2w_k:= (\dt)^{-2}(w_{k+1} - 2w_k + w_{k-1}).
\end{equation}

\subsection{Trapezoidal multistep method}
\label{subsec:trapezoidal}

Let us now describe our first fully discrete numerical scheme to solve problem \eqref{eq:truncated_weak_wave}. The space discretization is based on the finite element method on the truncated cylinder $\C_{\Y}$ described in Section \ref{subsec:fems}. The discretization in time is based on a trapezoidal multistep method.

The fully discrete scheme computes the sequence $V_{\dt} \subset \V(\T_{\Y})$, an approximation of the solution to \eqref{eq:truncated_weak_wave} at each time step. We initialize the scheme by setting 
\begin{equation}
\label{eq:time_disc_initeofd}
V_0 =G_{\T_{\Y}}\calH_{\alpha} g, \quad  V_1 =  G_{\T_{\Y}} \left(\mathcal{H}_{\alpha}g + \dt \mathcal{H}_{\alpha} h + \frac{1}{2}(\dt)^2 \partial_t^2\Ucal(0)\right),
\end{equation}
where $\Hcal_{\alpha}$ denotes the truncated $\alpha$--harmonic extension  and $\partial_t^2\Ucal(0) = \calH_\alpha w$ with $w \in \Hs$ satisfying 
\begin{equation}
\label{eq:Ucal0}
\langle w ,  \tr \phi \rangle = - a_\Y(\Ucal(0), \phi) + \langle f(0), \tr \phi \rangle \qquad \forall \phi \in \HL(y^{\alpha},\C_{\Y}).
\end{equation}
Note that, if $\tr \phi = 0$ the previous equation is satisfied for any $w$. 

For $k=1, \dots,K-1$, let $V_{k+1} \in \V(\T_{\Y})$ solve
\begin{equation}
\label{eq:time_disc_waveeofd}
\frac1{\dt^2} \langle \tr(V_{k+1}-2V_k+V_{k-1}), \tr W \rangle + a_\Y(\frakc V_k,W) = \langle \frakc f_k, \tr W \rangle
\end{equation}
for all $W\in \V(\T_{\Y})$, where $\frakc V_k$ and $\frakc f_k$ are defined in \eqref{eq:barwn}. To obtain an approximate solution to the fractional wave equation \eqref{eq:fractional_wave}, we define the sequence 
\begin{equation}
\label{eq:U_approximation}
 U_{\dt} = \{ U_k \}_{k=0}^{K} \subset S_0^1(\T,\Omega): \quad U_{\dt} := \tr V_{\dt}.
\end{equation}

\begin{remark}[locality]\rm
The main advantage of problem \eqref{eq:time_disc_initeofd}--\eqref{eq:time_disc_waveeofd} is that it provides an approximated solution to the fractional wave equation \eqref{eq:fractional_wave} based on the resolution of the \emph{local} elliptic problem
with a dynamic boundary condition \eqref{eq:truncated_weak_wave}.
\end{remark}


To present the stability of the scheme we introduce, for $k=1,\dots,K$, the unconditionally nonnegative discrete energy
\begin{equation}
 E_k(W_{\dt}) :=  \tfrac12\| \tr \frakd W_{k} \|_{L^2(\Omega)}^2  + \tfrac12\normC{W_{k-1/2}}^2.
\end{equation}

\begin{lemma}[energy conservation]
If $f \equiv 0$, the fully discrete scheme \eqref{eq:time_disc_initeofd}--\eqref{eq:time_disc_waveeofd} conserves energy, \ie for all $k \in  \{1,\dots,K\}$, we have that
\begin{equation}
\label{eq:energy_conserved}
E_k( V_{\dt} ) = E_1( V_{\dt} ).
\end{equation}
If $f \neq 0$, then, for $\ell \in \{1,\dots,K\}$, we have that
\begin{equation}
\label{eq:stability_trapezoidal}
E_{\ell}(V_{\dt})^{\frac{1}{2}} 
\leq  
E_{1}(V_{\dt})^{\frac{1}{2}} 
+ 
\frac{1}{\sqrt{2}} \sum_{k=1}^{\ell} \dt \| \frakc f_k \|_{L^2(\Omega)}.
\end{equation}
\label{lemma:sd}
In particular, we have that $E_{K}(V_{\dt})^{\frac{1}{2}} \leq E_{1}(V_{\dt})^{\frac{1}{2}} + \tfrac{1}{\sqrt{2}}\| \frakc f \|_{\ell^1(L^2(\Omega))}$.
\end{lemma}
\begin{proof}
Set $W=(2\dt)^{-1}(V_{k+1} - V_{k-1}) = 2^{-1}(\frakd V_{k+1} + \frakd V_k) = (\dt)^{-1}(V_{k+1/2}-V_{k-1/2})$ in \eqref{eq:time_disc_waveeofd}. Basic computations reveal that
\begin{equation}
\label{eq:step}
\frac{1}{\dt}\left( E_{k+1}(V_{\dt}) -E_k(V_{\dt})\right)
= \frac{1}{2}\langle \frakc f_k, \tr(\frakd V_{k+1} + \frakd V_{k}) \rangle.
\end{equation}
If $f \equiv 0$, the previous relation immediately yields \eqref{eq:energy_conserved}. If $f \neq 0$, an application of the Cauchy--Schwarz inequality allow us to conclude that
\begin{equation*}
E_{k+1}(V_{\dt}) -E_k(V_{\dt})
\leq 
\frac{\dt}{\sqrt{2}}  \| \frakc f_k  \|_{L^2(\Omega)}  \left(  E_{k+1}(V_{\dt})^{\frac{1}{2}} + E_{k}(V_{\dt})^{\frac{1}{2}}\right),
\end{equation*}
which yields
$
E_{k+1}(V_{\dt})^{\frac{1}{2}} -E_k(V_{\dt})^{\frac{1}{2}}
\leq \frac{\dt}{\sqrt{2}}  \|\frakc f_k  \|_{L^2(\Omega)}.
$
Adding over $\ell$ we arrive at the desired estimate \eqref{eq:stability_trapezoidal}. This concludes the proof.
\end{proof}
Let us now show the stability of the scheme.
\begin{lemma}[stability]
The fully discrete scheme \eqref{eq:time_disc_initeofd}--\eqref{eq:time_disc_waveeofd} is stable, namely, for $\ell \in \{1,\dots,K\}$, we have that
\[
\| \tr \frakd V_\ell \|_{L^2(\Omega)}  + \normC{V_{\ell-1/2}} \lesssim \| \tr \frakd V_1 \|_{L^2(\Omega)}  + \normC{V_{1/2}} + \sum_{k=1}^{\ell} \dt \| \frakc f_k \|_{L^2(\Omega)},
\]
where the hidden constant is independent of $V_{\dt}$ and $\dt$.
\end{lemma}
\begin{proof}
The proof follows immediately from \eqref{eq:stability_trapezoidal}.
\end{proof}

Let us now present an error analysis for the fully discrete scheme \eqref{eq:time_disc_initeofd}--\eqref{eq:time_disc_waveeofd}. To accomplish this task, we introduce the error $e_{\dt}:= V_{\dt} - \Ucal_{\dt}$, and write, as usual
\begin{equation}
\label{eq:spliting}
 e_{\dt} =  (V_{\dt} - G_{\T_{\Y}} \Ucal_{\dt}) + (G_{\T_{\Y}} \Ucal_{\dt}-\Ucal_{\dt}) =: \Theta_{\dt} + P_{\dt}.
\end{equation}
The control of $P_{\dt}$ follows from \eqref{eq:G_approx_1} and \eqref{eq:G_approx_3}: For $\ell \in \{0,1,2\}$, we have that 
\begin{equation}
\label{eq:nabla_P}
 \| \partial_t^{\ell} \nabla P_{\dt} \|_{\ell^2(L^2(y^{\alpha},\C_{\Y}))} \lesssim h_{\T} \mathfrak{A}(f,g,h)
\end{equation}
and
\begin{equation}
\label{eq:tr_P}
 \| \partial_t^{\ell} \tr P_{\dt} \|_{\ell^2(L^2(\Omega))} \lesssim h_{\T}^{1+s}\mathfrak{A}(f,g,h),
\end{equation}
where $\mathfrak{A}(f,g,h)$ is defined in \eqref{eq:frakA}. Notice that to obtain the estimates \eqref{eq:nabla_P} and \eqref{eq:tr_P} the regularity estimates of Corollary \ref{cor:st_reg_u} are essential.

In what follows we bound the sequence $\Theta_{\dt}$.

\begin{lemma}[error estimate for $\Theta_{\dt}$]
Let $\Ucal$ be the solution to \eqref{eq:truncated_weak_wave} and let $V_{\dt}$ be its fully discrete approximation defined as the solution to \eqref{eq:time_disc_initeofd}--\eqref{eq:time_disc_waveeofd}. If $\mathfrak{A}(f,g,h) < \infty$ and $\Xi(f,g,h) < \infty$, then 
\begin{equation}
 E_K(\Theta_{\dt})^{\frac{1}{2}} \lesssim h_{\T}^{1+s} \mathfrak{A}(f,g,h) 
 + (\dt)^2 \Xi(f,g,h),
 \label{eq:error_estimate_Theta}
\end{equation}
where $\Theta_{\dt} = V_{\dt} - G_{\T_{\Y}} \Ucal_{\dt}$, $\mathfrak{A}(f,g,h)$ and $\Xi(f,g,h)$ are defined in \eqref{eq:frakA} and \eqref{eq:Sigmanew}, respectively, and the hidden constant is independent of $V_{\dt}$, $\Ucal$, $\dt$, and $h_{\T}$.
\label{lemma:error_estimate_Theta}
\end{lemma}
\begin{proof}
We proceed in three steps.

Step 1. We invoke the continuous problem \eqref{eq:truncated_weak_wave}, the discrete equation \eqref{eq:time_disc_waveeofd}, and the definition of $G_{\T_\Y}$, given by \eqref{eq:elliptic_projection}, to arrive at the problem that controls the error:
For $k = 1,\dots,K-1$, $\Theta_{k+1} \in \V(\T_{\Y})$ solves
\begin{multline}
\frac1{\dt^2} \langle \tr(\Theta_{k+1}-2\Theta_k+\Theta_{k-1}), \tr W \rangle + a_\Y(\frakc \Theta_k,W) 
\\
= \langle \tr [ \mathfrak{c} \partial_t^2 \Ucal(t_k) - \frakd^2G_{\T_{\Y}} \Ucal(t_k)], \tr W \rangle \quad \forall W \in \V(\T_{\Y}),
\label{eq:Theta_equation}
\end{multline}
where $\mathfrak{c} \partial_t^2 \Ucal(t_k)$ and $\frakd^2G_{\T_{\Y}} \Ucal(t_k)$ are defined by \eqref{eq:barwn} and \eqref{eq:partial_discrete_2}, respectively. On the other hand, in view of Remark \ref{rem:initial_data_truncated} and \eqref{eq:time_disc_initeofd}, we have that
\begin{equation*}
 \Theta_0 = G_{\T_{\Y}}( \mathcal{H}_{\alpha} g - \Ucal(0)) = 0, \quad \Theta_1 =  G_{\T_{\Y}} \left(\mathcal{H}_{\alpha}g + \dt \mathcal{H}_{\alpha} h + \frac{1}{2}(\dt)^2 \partial_t^2\Ucal(0) - \Ucal(t_1)\right).
\end{equation*}

Now, we write, for $k \geq 1$, the difference $\mathfrak{c} \partial_t^2 \Ucal(t_k) - \frakd^2G_{\T_{\Y}} \Ucal(t_k)$ as follows:
\begin{align*}
\mathfrak{c} \partial_t^2 \Ucal(t_k) - \frakd^2G_{\T_{\Y}} \Ucal(t_k)
=  \left[ \partial_t^2 \Ucal(t_k) - \frakd^2 \Ucal(t_k)\right]
+ \left[\frakd^2 \Ucal(t_k) - \frakd^2G_{\T_{\Y}} \Ucal(t_k)\right]
\\
+ \tfrac{\dt}{4} \left [\frakd \partial_t^2 \Ucal(t_{k+1}) - \frakd \partial_t^2 \Ucal(t_{k}) \right] =: \mathrm{I}_k + \mathrm{II}_k + \mathrm{III}_k.
\end{align*}
We thus apply the stability estimate \eqref{eq:stability_trapezoidal} to \eqref{eq:Theta_equation} and obtain that
\begin{equation}
\label{eq:estimate_for_EK}
 E_K(\Theta_{\dt})^{\frac{1}{2}} \leq E_1(\Theta_{\dt})^{\frac{1}{2}} + \frac{1}{\sqrt{2}} \| \delta_{\dt} \|_{\ell^1(L^2(\Omega))},
\end{equation}
where $\delta_{\dt} = \{ \delta_k \}_{k=1}^{K-1}$ and $\delta_k = \tr [ \mathfrak{c} \partial_t^2 \Ucal(t_k) - \frakd^2G_{\T_{\Y}} \Ucal(t_k)]$.

Step 2. We proceed to control the term $\| \delta_{\dt} \|_{\ell^1(L^2(\Omega))}$. First, notice that
\begin{equation}
\label{eq:sum_k}
 \| \delta_{\dt} \|_{\ell^1(L^2(\Omega))}\leq \sum_{k=1}^{K-1} \dt \left( \| \tr \mathrm{I}_k\|_{L^2(\Omega)} + \|\tr  \mathrm{II}_k\|_{L^2(\Omega)} + \| \tr \mathrm{III}_k\|_{L^2(\Omega)} \right).
\end{equation}

To control $\| \tr \mathrm{I}_k \|_{L^2(\Omega)}$ we employ a basic result based on Taylor's Theorem. In fact, for $k \geq 1$, we have that 
\begin{equation}
\label{eq:I_k}
 \| \tr \mathrm{I}_k\|_{L^2(\Omega)} \lesssim (\dt)^2 \sup_{z 
 } \| \tr \partial_t^4 \Ucal(\cdot,z) \|_{L^2(\Omega)}.
\end{equation}

Now, notice that in view of \eqref{eq:spliting} we have that $\mathrm{II}_k = -\frakd^2 P(t_k)$. The same argument that yields \eqref{eq:I_k} allow us to conclude the estimate
\begin{equation*}
 \|\tr \mathrm{II}_k\|_{L^2(\Omega)} 
 \lesssim \| \tr \partial_t^2 P(t_k) \|_{L^2(\Omega)} 
 + (\dt)^2\sup_{z}  \| \tr \partial_t^4 P(\cdot,z) \|_{L^2(\Omega)}.
\end{equation*}
We invoke the trace estimate \eqref{Trace_estimate} and the stability estimate \eqref{eq:G_stable} of the weighted elliptic projection to conclude, for $z \in (0,T)$, that
\[
  \| \tr \partial_t^4( \Ucal - G_{\T_{\Y}} \Ucal)(\cdot,z) \|_{L^2(\Omega)} 
  \lesssim \| \tr \partial_t^4 \Ucal(\cdot,z) \|_{L^2(\Omega)} + \| \nabla \partial_t^4 \Ucal(\cdot,z) \|_{L^2(y^{\alpha},\C_{\Y})}.
\]
Consequently, an application, again, of the trace estimate allows us to conclude that
\begin{equation}
  \|\tr \mathrm{II}_k\|_{L^2(\Omega)} \lesssim h_{\T}^{1+s} \mathfrak{A}(f,g,h)
  +(\dt)^2 \| \nabla \partial_t^4 \Ucal \|_{L^{\infty}(0,T;L^2(y^{\alpha},\C_{\Y})},
   \label{eq:II_k}
\end{equation}
where we have used \eqref{eq:tr_P} with $\ell = 2$; $\mathfrak{A}(f,g,h)$ is defined in \eqref{eq:frakA}.

We finally bound $\mathrm{III}_k$. To accomplish this task, we invoke an argument based on Taylor's Theorem. In fact, for $k \geq 1$, we have that
\begin{multline}
 \|\tr \mathrm{III}_k\|_{L^2(\Omega)} 
 = \tfrac{\dt}{4} \|\tr[ \partial_t^3 \Ucal(\cdot,t_{k+1}) + \tfrac{\dt}{2} \partial_t^4 \Ucal(\cdot,\overline z) -  \partial_t^3 \Ucal(\cdot,t_{k}) - \tfrac{\dt}{2} \partial_t^4 \Ucal(\cdot,\underline z)  ]\|_{L^2(\Omega)}
 \\
 \lesssim (\dt)^2  \| \tr \partial_t^4 \Ucal \|_{L^{\infty}(0,T;L^2(\Omega))},
 \label{eq:III_k}
\end{multline}
where $\underline z$ and $\overline z$ belong to $(t_{k-1},t_{k+1})$.

Replacing the estimates \eqref{eq:I_k}, \eqref{eq:II_k}, and \eqref{eq:III_k} into \eqref{eq:sum_k} we arrive at
\begin{equation}
 \| \delta_{\dt} \|_{\ell^1(L^2(\Omega))} \lesssim h_{\T}^{1+s} \mathfrak{A}(f,g,h) 
 + (\dt)^2 \|\partial_t^4 \nabla \Ucal \|_{L^{\infty}(0,T;L^2(y^{\alpha},\C_{\Y}))}.
 \label{eq:estimation_for_delta}
\end{equation}

Step 3. We bound $E_1(\Theta_{\dt})$. Since $\Theta_0 = 0$, we utilize \eqref{eq:frakd} and write
\[
 E_1(\Theta_{\dt}) = \frac{1}{2} \| \tr \frakd \Theta_1 \|_{L^2(\Omega)}^2 +  \frac{1}{2} \normC{ \Theta_{1/2}}^2 = \frac{1}{2 (\dt)^2} \| \tr \Theta_1 \|^2_{L^2(\Omega)} + \frac{1}{8} \normC{ \Theta_1 }^2.
\]
In view of the trace estimate \eqref{Trace_estimate} and the stability property of $G_{\T_{\Y}}$ given in \eqref{eq:G_stable} we can thus conclude that
\begin{equation}
 \| \tr \Theta_1 \|_{L^2(\Omega)} 
 \lesssim  \normC{ \Theta_1 } 
 \lesssim \normC{ \Ucal(0) + \dt \Ucal_t (0) + \tfrac{1}{2}(\dt)^2 \partial_t^2\Ucal(0) - \Ucal(t_1)}.
\label{eq:Estimate_for_Theta1}
\end{equation}
An application of Taylor's Theorem reveals that
\[
  \| \tr \Theta_1 \|_{L^2(\Omega)} \lesssim \normC{ \Theta_1 } \lesssim (\dt)^3 \| \partial_t^3 \nabla \Ucal \|_{L^{\infty}(0,T;L^2(y^{\alpha},\C_{\Y}))},
\]
which immediately yields 
\begin{equation}
\label{eq:estimation_for_E1}
E_1(\Theta_{\dt})^{\frac{1}{2}} \lesssim (\dt)^2 \| \partial_t^3 \nabla \Ucal \|_{L^{\infty}(0,T;L^2(y^{\alpha},\C_{\Y}))}.
\end{equation}

The desired estimate follows from replacing \eqref{eq:estimation_for_delta} and \eqref{eq:estimation_for_E1} into \eqref{eq:estimate_for_EK} and using the time--regularity results of Theorem \ref{thm:time_regularity}. 
\end{proof}

The exponential error estimate of Lemma \ref{lemma:exp_estimate} combined with the error estimates of Lemma \ref{lemma:error_estimate_Theta} allow us to conclude the following error estimates.

\begin{lemma}[error estimates for \eqref{eq:time_disc_initeofd}--\eqref{eq:time_disc_waveeofd}]
Let $\ue$ be the solution to \eqref{eq:weak_wave} and let $V_{\dt}$ be the solution to the fully discrete problem \eqref{eq:time_disc_initeofd}--\eqref{eq:time_disc_waveeofd}.  If $\mathfrak{A}(f,g,h) < \infty$ and $\Xi(f,g,h) < \infty$,  we have the following error estimates:
\begin{equation}
\| \tr( \partial_t \ue(t_{K-1/2}) - \frakd V_K) \|_{L^2(\Omega)} \lesssim h^{1+s}_{\T} \mathfrak{A}(f,g,h) 
\\ 
 + (\dt)^2 \Xi(f,g,h),
\label{eq:error_estimate_total1}
\end{equation}
and
\begin{equation}
\normC{\ue(t_{K-1/2}) - V_{K-1/2}} \lesssim h_{\T} \mathfrak{A}(f,g,h) + (\dt)^2 \Xi(f,g,h),
 \label{eq:error_estimate_total2}
\end{equation}
where $\mathfrak{A}(f,g,h)$ and $\Xi(f,g,h)$ are defined by \eqref{eq:frakA} and \eqref{eq:Sigmanew}, respectively, and the hidden constants are independent of $V_{\dt}$, $\Ucal$, $\ue$, $\dt$, and $h_{\T}$.
\label{lemma:error_estimate_total}
\end{lemma}
\begin{proof}
We proceed in  several steps.

Step 1. We begin with the following trivial application of the triangle inequality:
\begin{multline}
\| \tr [ \partial_t \ue(t_{K-1/2}) - \frakd V_K] \|_{L^2(\Omega)} \leq   
\| \tr \partial_t [ \ue(t_{K-1/2}) -  \Ucal(t_{K-1/2}) ] \|_{L^2(\Omega)}   
\\
  + 
\| \tr [\partial_t \Ucal(t_{K-1/2})- \frakd V_K] \|_{L^2(\Omega)} =: \textrm{I} + \textrm{II}.
\label{eq:AUX}
\end{multline}
To control the term $ \textrm{I}$ we invoke the exponential error estimate \eqref{eq:exp_convergence}. The latter yields
\[
\textrm{I} = \left \| \tr  \partial_t[ \ue(t_{K-1/2}) -  \Ucal(t_{K-1/2})] \right \|_{L^2(\Omega)} \lesssim e^{-\sqrt{\lambda_1} \Y/2} \Sigma_1(f,g,h).
\]
The control of the term $\textrm{II}$ is as follows:
\begin{equation*}
 \textrm{II} 
\leq \left \| \tr[ \partial_t \Ucal(t_{K-1/2})- \frakd\Ucal_K]  \right \|_{L^2(\Omega)}  + \left \| \tr \frakd e_K \right \|_{L^2(\Omega)} =: \textrm{II}_1 + \textrm{II}_2,
\end{equation*}
recalling that $e_K = V_K - \Ucal_K$. Replace the obtained estimates into \eqref{eq:AUX}. This yields
\begin{equation}
\| \tr[ \partial_t \ue(t_{K-1/2}) - \frakd V_K] \|_{L^2(\Omega)} \lesssim e^{-\sqrt{\lambda_1} \Y/2} 
 \Sigma_1(f,g,h) + \textrm{II}_1 + \textrm{II}_2.
  \label{eq:AUXAUX}
\end{equation}

Step 2.  The control of $\textrm{II}_1 = \|\tr[ \partial_t \Ucal(t_{K-1/2})- \frakd\Ucal_K] \|_{L^2(\Omega)}$ follows from a simple application of Taylor's Theorem. In fact, we have that
\begin{equation}
\nonumber
\textrm{II}_1 
= \left \| \tr \left( \partial_t \Ucal(t_{K-1/2})- \tfrac{ \Ucal_K - \Ucal_{K-1}} {\dt} \right) \right \|_{L^2(\Omega)}
\lesssim (\dt)^2 \sup_{z} \| \tr \partial_t^3 \Ucal(\cdot,z) \|_{L^2(\Omega)}.
\label{eq:AUX2}
\end{equation}
We now focus on the term  $\mathrm{II}_2 = \| \tr \frakd e_K \|_{L^2(\Omega)}$. The triangle inequality yields
\[
  \mathrm{II}_2  \leq  \| \tr \frakd \Theta_K \|_{L^2(\Omega)}  +  \| \tr \frakd P_K \|_{L^2(\Omega)}.
\]
The results of Lemma \ref{lemma:error_estimate_Theta} imply that
\begin{equation}
 \| \tr \frakd \Theta_K \|_{L^2(\Omega)} \lesssim h_{\T}^{1+s} \mathfrak{A}(f,g,h) 
 \\
 + (\dt)^2 \Xi(f,g,h).
 \label{eq:Theta_Aux}
\end{equation}
The control of $\| \tr \frakd P_K \|_{L^2(\Omega)}$ follows the same arguments used to bound $\textrm{II}_1:$
\begin{align*}
 \| \tr \frakd P_K \|_{L^2(\Omega)} & = (\dt)^{-1} \| \tr (P_K - P_{K-1}) \|_{L^2(\Omega)}
 \\
 & \lesssim \| \tr \partial_t P(t_{K-1/2}) \|_{L^2(\Omega)} + (\dt)^2 \| \tr \partial_t^3 P  \|_{L^{\infty}(0,T;L^2(\Omega))}.
\end{align*}
The stability property \eqref{eq:G_stable} and the estimate \eqref{eq:tr_P} imply the estimates 
\begin{align*}
 \| \tr \frakd P_K \|_{L^2(\Omega)} & \lesssim \| \tr \partial_t P \|_{L^{\infty}(0,T;L^2(\Omega))} + (\dt)^2  \| \nabla \partial_t^3 \Ucal \|_{L^{\infty}(0,T;L^2(y^{\alpha},\C_{\Y})}
 \\
 & \lesssim h_{\T}^{1+s}\mathfrak{A}(f,g,h) + (\dt)^2  \| \nabla \partial_t^3 \Ucal \|_{L^{\infty}(0,T;L^2(y^{\alpha},\C_{\Y})}.
\end{align*}
The previous estimate combined with \eqref{eq:Theta_Aux} allow us to control $\mathrm{II}_2$. Replacing the obtained estimate for  $\mathrm{II}_1$ and $\mathrm{II}_2$ into \eqref{eq:AUXAUX} yield the desired estimate \eqref{eq:error_estimate_total1}.  

Step 3. To obtain  \eqref{eq:error_estimate_total2} we invoke similar arguments upon using
the estimate
\[
  \normC{P_{K-1/2}} \lesssim  \normC{P_{K}} +  \normC{P_{K-1}} \lesssim \| \nabla P_{\dt}\|_{\ell^{\infty}(L^2(y^{\alpha},\C_{\Y}))} \lesssim h_{\T} \mathfrak{A}(f,g,h).
\]
This concludes the proof.
\end{proof}

The following error estimates follow immediately from Lemma \ref{lemma:error_estimate_total} and show how the fully discrete approximation $U_{\dt}$ approximates $u$.

\begin{theorem}[error estimates for \eqref{eq:U_approximation}]
Let $u$ be the solution to \eqref{eq:fractional_wave} and let $U_{\dt}$ be its fully discrete approximation defined by \eqref{eq:U_approximation}. If $\mathfrak{A}(f,g,h) < \infty$ and $\Xi(f,g,h) < \infty$, then 
\begin{equation}
 \| \partial_t u(t_{K-1/2}) - \frakd U_K \|_{L^2(\Omega)}  \lesssim h^{1+s}_{\T} \mathfrak{A}(f,g,h) 
 + (\dt)^2 \Xi(f,g,h),
  \label{eq:error_estimate_total1s}
  \end{equation}
  and
  \begin{equation}
   \| u(t_{K-1/2}) - U_{K-1/2}  \|_{\Hs}  \lesssim h_{\T} \mathfrak{A}(f,g,h) 
 + (\dt)^2 \Xi(f,g,h),
 \label{eq:error_estimate_total2s}
  \end{equation}
where $\mathfrak{A}(f,g,h)$ and $\Xi(f,g,h)$ are defined by \eqref{eq:frakA} and \eqref{eq:Sigmanew}, respectively, and the hidden constants are independent of $U_{\dt}$, $u$, $\dt$, and $h_{\T}$.
\label{lemma:error_estimate_totals}
\end{theorem}
\subsection{The leapfrog scheme}
\label{subsec:leapfrog}

We now present a second fully discrete scheme to approximate the solution to \eqref{eq:fractional_wave}. To advance in time we use the leapfrog time-stepping method while the discretization in space is based on the finite element method described in section \ref{subsec:fems}. The scheme computes a sequence $V_{\dt} \subset \V(\T_{\Y})$, an approximation to the solution to \eqref{eq:truncated_weak_wave} at each time step. To begin with the description of the scheme, we first initialize it by setting
\begin{equation}
\label{eq:time_disc_initlf}
V_0 = G_{\T_{\Y}} \calH_{\alpha} g, \quad  V_1 = G_{\T_{\Y}} \left(\calH_{\alpha} g+ \dt \calH_{\alpha}h + \frac{1}{2}( \dt )^2 \partial_t^2 \Ucal(0) \right),
\end{equation}
where $\partial_t^2\Ucal(0) = \calH_{\alpha}w$ and $w$ solves \eqref{eq:Ucal0}. For $k=1,\dots,K -1$, $V_{k+1} \in \HL(y^{\alpha},\C_{\Y})$ solves 
\begin{equation}
\label{eq:time_disc_wavelf}
\frac1{\dt^2} \langle \tr(V_{k+1}-2V_k+V_{k-1}), \tr W \rangle + a_\Y(V_k,W) = \langle f_k, \tr W \rangle
\end{equation}
for all $W \in \V(\T_{\Y})$. As in the previous section, we define an approximated solution to problem \eqref{eq:fractional_wave} as 
\begin{equation}
\label{eq:U_approximationfrog}
U_{\dt} = \{ U_k \}_{k=0}^{K} \subset S_0^1(\Omega,\T), \quad U_{\dt} = \tr V_{\dt}.
\end{equation}

In the analysis that follows, the following discrete inverse inequality will be instrumental.

\begin{lemma}[discrete inverse inequality]
Let $\eta \in \Hsd$ and let $X \in \V(\T_{\Y})$ be the solution to
\begin{equation}
 a_\Y(X,W) = \langle \eta, \tr W \rangle  \quad \forall W \in \V(\T_{\Y}).
 \label{eq:probX}
\end{equation}
We thus have that $\normC{ X }  \lesssim \|\tr X \|_{\Hs}$ and that
\begin{equation}
\label{eq:inverse_inequality}
\normC{ X }  \leq C_{\rm inv} h_\T^{-s} \|\tr X\|_{L^2(\Omega)},
\end{equation}
for some constant $ C_{\rm inv} > 0$
\end{lemma}
\begin{proof}
We begin by defining $Z := \tr X \in S_0^1(\T,\Omega)$. There exists $\chi \in \HL(y^{\alpha},\C_{\Y})$ such that $\tr \chi = Z$. In fact $\chi = \calH_{\alpha} Z$, where $\calH_{\alpha}$ is defined in \eqref{eq:calH}.

Let us introduce the operator $\tilde\Pi = \Pi_{x'} \otimes \tilde\Pi_{y}^ {\bmr }$, where $\tilde\Pi_{y}^ {\bmr }$ is a slight modification of the operator of \cite[Section 5.5.1]{BMNOSS:17}: on the first interval $I_1$, interpolation at the edge point $0$ is used rather than in the middle point of $I_1$. The operator $\Pi_{x'}$ corresponds to the $L^2(\Omega)$--orthogonal projection operator defined in \eqref{eq:ortogonal_projection}. Define $\tilde X = \tilde \Pi \chi$ and notice that the stability properties of $\tilde \Pi$ yield
\[
\normC{\tilde X} \lesssim \normC{\chi}.
\]
This, in view of the fact that $ \normC{\chi} \lesssim  \| Z \|_{\Hs}$, implies $\normC{\tilde X} \lesssim \| Z \|_{\Hs}$.

Now, since $Z \in S_0^1(\T,\Omega)$, we have that $\tr \tilde X = \tr \chi = Z = \tr X$ and then that
$\tr(X-\tilde X) = 0$. Since $X-\tilde X \in \V(\T_{\Y})$, we can thus invoke problem \eqref{eq:probX} and conclude that
\[
 a_{\Y}(X,X-\tilde X) = \langle \eta, \tr (X- \tilde X) \rangle = 0,
\]
which yields
\begin{align*}
 a_{\Y}(\tilde X, \tilde X) & = a_{\Y}((\tilde X-X) + X,(\tilde X-X) +  X) 
 \\
 & = a_{\Y}((\tilde X- X),(\tilde X- X)) + a_{\Y}(X,X) \geq a_{\Y}(X,X).
\end{align*}
This immediately implies that
$
 \normC{X} \leq \normC{\tilde X},
$
and thus, since $\normC{\tilde X} \lesssim \| Z \|_{\Hs}$, that
\[
 \normC{X}  \lesssim \| Z \|_{\Hs}.
\]
Since $Z = \tr X$, we have thus obtained the desired estimate $\normC{X} \lesssim \| \tr X \|_{\Hs}$. The estimate \eqref{eq:inverse_inequality} thus follows from, for instance, the results of \cite{MR2047080}.
\end{proof}

To analyze the fully discrete scheme \eqref{eq:time_disc_initlf}--\eqref{eq:time_disc_wavelf}, we define, for $k= 1,\dots, K$, the discrete energy
\begin{equation}
 \Ener_k(W_{\dt}) :=  \tfrac12\| \tr \frakd W_{k} \|_{L^2(\Omega)}^2  + \tfrac12a_{\Y}(W_{k},W_{k-1}),
\end{equation}
where the bilinear form $a_{\Y}$ is defined in \eqref{eq:a_Y}.

In the result that follows we show the nonnegativity of the discrete energy $ \Ener_k$ under the following CFL condition: $\dt$ is chosen to be sufficiently small such that
\begin{equation}
\label{eq:CFL}
1 - C^{^2}_{\mathrm{inv}}\frac{(\dt)^2}{2h_{\T}^{2s}} \geq \theta > 0, \quad \theta \in (0,1).
\end{equation}
The constant $C_{\mathrm{inv}}$ is as in \eqref{eq:inverse_inequality}.

\begin{lemma}[CFL condition and nonnegativity of $\Ener_k$]
If \eqref{eq:CFL} holds, then
\begin{equation}
\label{eq:Ek_is_positive}
 \Ener_k(V_{\dt}) \geq \frac{\theta}{2} \| \tr \frakd V_{k} \|_{L^2(\Omega)}^2 +\frac{1}{4} \left[ \normC{V_k}^2 + \normC{V_{k-1}}^2 \right] \geq 0
\end{equation}
for all $k \in \{ 1,\cdots,K\}$.
\label{lemma:lf}
\end{lemma}
\begin{proof}
We invoke the inverse inequality \eqref{eq:inverse_inequality} and thus the CFL condition \eqref{eq:CFL} to arrive at
\begin{align*}
2a_{\Y} (V_k,V_{k-1}) & = \normC{V_k}^2+ \normC{V_{k-1}}^2 - \normC{V_{k} - V_{k-1}}^2
\\
& \geq  \normC{V_k}^2 + \normC{V_{k-1}}^2 - C_{\rm{inv}}^2 h_\T^{-2s} \|\tr (V_{k} - V_{k-1})\|_{L^2(\Omega)}^2\\
& \geq \normC{V_k}^2 + \normC{V_{k-1}}^2 + 2(\theta-1)  \| \tr \frakd V_k \|^2_{L^2(\Omega)},
\end{align*}
where, in the last step, we have used definition \eqref{eq:frakd}. Consequently,
\[
  \Ener_k(V_{\dt}) \geq \frac12 \| \tr \frakd V_{k} \|_{L^2(\Omega)}^2 + \frac14\normC{V_k}^2 + \frac14\normC{V_{k-1}}^2 + \frac{(\theta-1)}{2}  \| \tr \frakd V_k \|^2_{L^2(\Omega)},
\]
which immediately yields \eqref{eq:Ek_is_positive}. This concludes the proof.
\end{proof}

\begin{lemma}[energy conservation]
If $f \equiv 0$, the fully discrete scheme \eqref{eq:time_disc_initlf}--\eqref{eq:time_disc_wavelf} conserves energy, \ie for all $k \in \{ 1, \dots, K\}$, we have that
\begin{equation}
\label{eq:energy_conservedfl}
 \Ener_k(V_{\dt}) = \Ener_1(V_{\dt}).
\end{equation}
If $f \neq 0$, then, for $\ell \in \{1,\dots,K \}$, we have that
\begin{equation}
\label{eq:stability_leapgrog}
 \Ener_{\ell}(V_{\dt})^{\frac{1}{2}} \leq  \Ener_{1}(V_{\dt})^{\frac{1}{2}} + \frac{1}{\sqrt{2\theta}} \sum_{k=1}^{\ell} \dt \|f_k \|_{L^2(\Omega)}.
\end{equation}
In particular, we have that $\Ener_{K}(V_{\dt})^{\frac{1}{2}} \leq \Ener_{1}(V_{\dt})^{\frac{1}{2}} + \tfrac{1}{\sqrt{2\theta}}\| f \|_{\ell^1(L^2(\Omega))}$.
\end{lemma}
\begin{proof}
Set $W=(2\dt)^{-1}(V_{k+1} - V_{k-1}) = 2^{-1}(\frakd V_{k+1} + \frakd V_k) 
$ in \eqref{eq:time_disc_wavelf}. This yields
\begin{equation}
\label{eq:stab_aux}
  \frac{1}{\dt}\left( \Ener_{k+1}(V_{\dt}) -  \Ener_{k}(V_{\dt})  \right) = \frac{1}{2}\langle f_k, \tr(\frakd V_{k+1} + \frakd V_k)  \rangle.
\end{equation}
In the case that $f \equiv 0$, \eqref{eq:stab_aux} immediately yields \eqref{eq:energy_conservedfl}. If $f \neq 0$, a trivial application of the Cauchy--Schwarz inequality reveals that
\[
 \Ener_{k+1}(V_{\dt}) -  \Ener_{k}(V_{\dt})  \leq \frac{\dt}{2} \| f_k \|_{L^2(\Omega)} \left( \| \tr \frakd V_{k+1}\|_{L^2(\Omega)} + \| \frakd V_k \|_{L^2(\Omega)} \right).
\]
Invoke the estimate \eqref{eq:Ek_is_positive} and conclude, for $k \in \{1,\dots,K-1\}$, that $\Ener_{k}(V_{\dt}) \geq (\theta/2) \| \tr \frakd V_k\|^2_{L^2(\Omega)}$. Thus, 
\[
 \Ener_{k+1}(V_{\dt}) -  \Ener_{k}(V_{\dt})  \leq \frac{\dt}{\sqrt{2\theta}} \| f_k \|_{L^2(\Omega)} \left(  \Ener_{k+1}^{\frac{1}{2}}(V_{\dt}) +  \Ener_{k}^{\frac{1}{2}}(V_{\dt}) \right),
\]
Consequently, we arrive at $\Ener_{k+1}(V_{\dt})^{\frac{1}{2}} -  \Ener_{k}(V_{\dt})^{\frac{1}{2}} \leq (\dt/\sqrt{2\theta}) \| f_k \|_{L^2(\Omega)}$ which, by adding over $\ell$, yields \eqref{eq:stability_leapgrog}. This concludes the proof.
\end{proof}

\begin{lemma}[stability]
The fully discrete scheme \eqref{eq:time_disc_initlf}--\eqref{eq:time_disc_initlf} is stable: for $\ell \in \{1,\dots,K\}$, we have that
\begin{equation}
\label{eq:stab_leapgrog}
\| \tr \frakd V_\ell \|_{L^2(\Omega)}  + \normC{V_{\ell}} \lesssim \| \tr \frakd V_1 \|_{L^2(\Omega)}  + \normC{V_{0}} + \normC{V_{1}} + \sum_{k=1}^{\ell} \dt \| f_k \|_{L^2(\Omega)},
\end{equation}
where the hidden constant is independent of $V_{\dt}$, $\dt$ and $h_{\T}$ but depends on $\theta$.
\end{lemma}
\begin{proof}
We begin by noticing that, in view of \eqref{eq:Ek_is_positive}, we arrive at
\begin{equation}
\label{eq:stab_aux_2}
 \| \tr \frakd V_\ell \|_{L^2(\Omega)}  + \normC{V_{\ell}} \lesssim ( \theta^{-1/2} +1) \Ener_{l}(V_{\dt})^{\frac{1}{2}}.
\end{equation}

Now, since 
\[
2a_{\Y}(V_1,V_0) = a_{\Y}(V_1,V_1) + a_{\Y}(V_0,V_0) - a_{\Y}(V_1-V_0,V_1-V_0) \leq a_{\Y}(V_1,V_1) + a_{\Y}(V_0,V_0),   
\]
an application of the estimate \eqref{eq:stability_leapgrog} allows us to conclude that
\[
  \Ener_{\ell}(V_{\dt})^{\frac{1}{2}} \lesssim  \| \tr \frakd V_1 \|_{L^2(\Omega)} + \normC{V_{0}} + \normC{V_{1}} + \tfrac{1}{\sqrt{\theta}} \sum_{k=1}^{\ell} \dt \|f_k \|_{L^2(\Omega)}.
\]
The desired estimate \eqref{eq:stab_leapgrog} thus follows from replacing the previous estimate into \eqref{eq:stab_aux_2}. This concludes the proof.
\end{proof}

We now present error estimates for the fully discrete approximation $U_{\dt}$ defined in \eqref{eq:U_approximationfrog} that is based on the solution $V_{\dt}$ to the fully discrete scheme \eqref{eq:time_disc_initlf}--\eqref{eq:time_disc_wavelf}. The arguments are similar to the ones used to prove the results in Lemma \ref{lemma:error_estimate_Theta}, Lemma \ref{lemma:error_estimate_total}, and Theorem \ref{lemma:error_estimate_totals}. For brevity we leave details to the reader.

\begin{theorem}[error estimates for \eqref{eq:U_approximationfrog}]
Let $u$ be the solution to \eqref{eq:fractional_wave} and let $U_{\dt}$ be its fully discrete approximation defined by \eqref{eq:U_approximationfrog}. If $\mathfrak{A}(f,g,h) < \infty$ and $\Xi(f,g,h) < \infty$, then 
\begin{equation}
 \| \partial_t u(t_{K-1/2}) - \frakd U_K \|_{L^2(\Omega)}  \lesssim h^{1+s}_{\T} \mathfrak{A}(f,g,h) 
 + (\dt)^2 \Xi(f,g,h),
  \label{eq:error_estimate_total1sfrog}
  \end{equation}
  and
  \begin{equation}
   \| u(t_K) - U_{K}  \|_{\Hs}  \lesssim h_{\T} \mathfrak{A}(f,g,h) 
 + (\dt)^2 \Xi(f,g,h),
 \label{eq:error_estimate_total2sfrog}
  \end{equation}
where $\mathfrak{A}(f,g,h)$ and $\Xi(f,g,h)$ are defined by \eqref{eq:frakA} and \eqref{eq:Sigmanew}, respectively, and the hidden constants are independent of $U_{\dt}$, $u$, $\dt$, and $h_{\T}$.
\label{lemma:error_estimate_totalsfrog}
\end{theorem}

\subsection{Computable data}

In \eqref{eq:time_disc_initeofd} we considered
\begin{equation}
V_0 = G_{\T_{\Y}} \calH_{\alpha} g, \quad  V_1 = G_{\T_{\Y}} \left(\calH_{\alpha} g+ \dt \calH_{\alpha}h + \frac{1}{2}( \dt )^2 \partial_t^2 \Ucal(0) \right),
 \label{eq:time_disc_initnew}
\end{equation}
as initial data for the fully discrete schemes of sections \ref{subsec:trapezoidal} and \ref{subsec:leapfrog}. Since the action of $\calH_{\alpha}$ involves the resolution of a problem posed on an infinite dimensional space, we immediately conclude that the initial data $V_0$ and $V_1$ \emph{are not computable}.

To overcome this deficiency, we introduce the discrete extension operator $\mathcal{H}_{\alpha}^{\T}$, which is defined as follows:
if $e \in S_0^1(\Omega,\T)$, then 
\[
\mathcal{H}_{\alpha}^{\T} e := E \in \V(\T_{\Y}): \quad
a_{\Y}(E,W) = 0 \quad \forall W \in  \V(\T_{\Y}): \tr W = 0 ,
\quad
\tr E =  e.
\]

With $\mathcal{H}_{\alpha}^{\T}$ at hand, we define the following computable initial data:
\begin{equation}
\label{eq:time_disc_initeofdnew}
\widetilde V_0 = \mathcal{H}_{\alpha}^{\T}  \Pi_{x'} g, \quad  \widetilde V_1 =  \mathcal{H}_{\alpha}^{\T} \left( \Pi_{x'}g + \dt \Pi_{x'} h  + \frac{1}{2}(\dt)^2 Z \right).
\end{equation}
$Z \in S_0^1(\T,\Omega)$ solves $\langle Z, \tr W \rangle = -a_{\Y}(\widetilde V_0,W) + \langle f(0), \tr W \rangle$ for all $W \in \V(\T_{\Y})$. Notice that $Z$ corresponds to a finite element approximation of $\tr \partial_t^2 \mathcal{U}(0)$.

If we consider $\widetilde V_0$ and $\widetilde V_1$, instead of $V_0$ and $V_1$, as initial data for the schemes of sections \ref{subsec:trapezoidal} and \ref{subsec:leapfrog}, then, to provide an a priori error analysis, it is necessary to modify the first two elements of the sequence $\Theta_{\dt}$, defined in \eqref{eq:spliting}, as follows:
\begin{align*}
\Theta_0 = \widetilde V_0 - G_{\T_{\Y}} \mathcal{U}(t_0),
\qquad
\Theta_1 = \widetilde V_1 - G_{\T_{\Y}} \mathcal{U}(t_1).
\end{align*}
In particular, it suffices to estimate
\[
 E_1(\Theta_{\dt})^{\frac{1}{2}} =\left( \tfrac{1}{2} \| \tr \frakd \Theta_1 \|_{L^2(\Omega)}^2 + \tfrac{1}{2} \normC{\Theta_{1/2}}^2 \right)^{\frac{1}{2}}.
\]

We present the following error estimates.

\begin{lemma}[error estimates for $\Theta_0$ and $\Theta_1$]
If $(V_0,V_1)$ and $(\widetilde V_0, \widetilde V_1)$ are defined by \eqref{eq:time_disc_initnew} and \eqref{eq:time_disc_initeofdnew}, respectively, then
\begin{equation}
\normC{\Theta_0} \lesssim h_{\T}\| g \|_{\mathbb{H}^{1+s}(\Omega)},
\label{eq:V0-tildeV0}
\end{equation}
and
\begin{equation}
\normC{\Theta_1} \lesssim  h_{\T}\mathfrak{A}(f,g,h) + (\dt)^2 \Xi(f,g,h),
\label{eq:V1-tildeV1}
\end{equation}
where $\mathfrak{A}(f,g,h)$ and $\Xi(f,g,h)$ are defined by \eqref{eq:frakA} and \eqref{eq:Sigmanew}, respectively, and the
the hidden constants are independent of $g$, $h$, $(V_0,V_1)$, $(\widetilde V_0,\widetilde V_1)$, and $h_{\T}$.
\end{lemma}
\begin{proof}
Since $V_0, \tilde V_0 \in \V(\T_{\Y})$, we can invoke property \eqref{eq:elliptic_projection} and conclude that
\begin{align*}
\normC{\widetilde V_0-V_0}^2 &= a_\Y(\widetilde V_0-V_0, \widetilde V_0 - G_{\T_{\Y}}\calH_\alpha \Pi_{x'} g) + a_\Y(\widetilde V_0-V_0, G_{\T_{\Y}}\calH_\alpha (\Pi_{x'} g - g))
\\
& =
 a_\Y(\widetilde V_0-V_0, \widetilde V_0 - \calH_\alpha \Pi_{x'} g) + a_\Y(\widetilde V_0-V_0, \calH_\alpha( \Pi_{x'} g -g)) = \mathrm{I} + \mathrm{II}.
\end{align*}
To bound $\mathrm{I}$, we notice that $\tr ( \widetilde V_0 - \calH_\alpha \Pi_{x'} g) =0$. On the other hand, $V_0 - \widetilde V_0$ satisfies
\[
a_{\Y}( V_0 - \widetilde V_0, W) = 0  \quad \forall W \in  \V(\T_{\Y}): \tr W = 0 ,
\quad
\tr ( V_0 - \widetilde V_0) =  \tr G_{\T_{\Y}}\calH_\alpha g - \Pi_{x'}g.
\]
Consequently,
$
 \mathrm{I} = 0.
$
Now, since $\mathcal{H}_{\alpha}$ satisfies that $\normC{\mathcal{H}_{\alpha} w} \lesssim  \| w \|_{\Hs}$  for all $w \in \Hs$, we arrive at
\begin{equation*}
|\mathrm{II}| \lesssim \normC{\widetilde V_0-V_0}  \| g - \Pi_{x'} g\|_{\Hs}
\lesssim h_{\T}  \normC{V_0-V_0}  \| g \|_{\mathbb{H}^{1+s}(\Omega)}.
\end{equation*}
Since $\mathrm{I} = 0$, the estimate for $\mathrm{II}$ yields 
\eqref{eq:V0-tildeV0}.

We now control $\normC{\Theta_1}$. A basic application of the triangle inequality together with estimate \eqref{eq:Estimate_for_Theta1} reveal that
\begin{align*}
\normC{\Theta_1} & \leq \normC{\widetilde V_1 - V_1} + \normC{V_1 - G_{\T_{\Y}}\mathcal{U}(t_1)} 
\\
& \lesssim \normC{\widetilde V_1 - V_1} +  (\dt)^3 \| \partial_t^3 \nabla \Ucal \|_{L^{\infty}(0,T;L^2(y^{\alpha},\C_{\Y}))}.
\end{align*}
It thus suffices to bound $\normC{\widetilde V_1 - V_1}$. To accomplish this task, we first notice that a simple application of Taylor's Theorem yields
\begin{multline}
  \widetilde V_1 =  \mathcal{H}_{\alpha}^{\T} \Pi_{x'} \left[ \tr \mathcal{U}(0) + \dt  \tr \partial_t \mathcal{U}(0)  + \tfrac{(\dt)^2}{2} \tr \partial_t^2 \mathcal{U}(0) \right] 
  \\
  + \tfrac{(\dt)^2 }{2} \mathcal{H}_{\alpha}^{\T}(Z- \Pi_{x'} \tr \partial_t^2 \mathcal{U}(0)) = \mathcal{H}_{\alpha}^{\T}\Pi_{x'}\left[ \tr \mathcal{U}(t_1) - \tfrac{(\dt)^3}{6} \tr \partial_t^3\mathcal{U}(\overline \zeta)\right]
  \\
  + \tfrac{(\dt)^2 }{2} \mathcal{H}_{\alpha}^{\T}(Z- \Pi_{x'} \tr \partial_t^2 \mathcal{U}(0)),
\end{multline}
with $\overline \zeta \in (0,t_1)$. Similar arguments allow us to conclude that
\begin{equation}
  V_1 =  G_{\T_{\Y}} \left( \mathcal{U}(t_1) - \frac{(\dt)^3}{6} \partial_t^3\mathcal{U}(\underline \zeta)\right),
\end{equation}
with $\underline \zeta \in (0,t_1)$. Consequently,
\begin{multline}
 \normC{\widetilde V_1 - V_1} \lesssim \normC{ \mathcal{H}_{\alpha}^{\T}\Pi_{x'}\tr \mathcal{U}(t_1) -  G_{\T_{\Y}} \mathcal{U}(t_1) } + (\dt)^3 \| \partial_t^3\nabla \mathcal{U}\|_{L^{\infty}(0,T;L^2(y^{\alpha},\C))}
 \\
 + (\dt)^2 \normC{\mathcal{H}_{\alpha}^{\T}(Z- \Pi_{x'} \tr \partial_t^2 \mathcal{U}(0))} = \textrm{I + II + III}.
\end{multline}
To estimate $\mathrm{I}$ we invoke the same arguments that lead to 
\eqref{eq:V0-tildeV0}:
\[
\mathrm{I} \lesssim h_{\T} \| \tr \mathcal{U}(t_1) \|_{\mathbb{H}^{1+s}(\Omega)} \lesssim  h_{\T} \| \tr \mathcal{U} \|_{L^{\infty}(0,T;\mathbb{H}^{1+s}(\Omega))}.
\]
A bound for the term $\mathrm{III}$ follows from stability results. The collection of these estimates yield \eqref{eq:V1-tildeV1}
\end{proof}

\begin{lemma}[estimate for $E_1(\Theta_{\dt})$]
If we consider $\widetilde V_0$ and $\widetilde V_1$ as initial data for the schemes of sections \ref{subsec:trapezoidal} and \ref{subsec:leapfrog}, we then have that
\begin{equation}
 \| \tr \frakd\Theta_1 \|_{L^2(\Omega)} = \frac{1}{\dt}  \| \tr (\Theta_1 - \Theta_0) \|_{L^2(\Omega)} \lesssim h^{1+s}_{\T}\mathfrak{A}(f,g,h) + (\dt)^2 \Xi(f,g,h),
 \label{eq:estimatefrakd12}
\end{equation}
and
\begin{equation}
  \normC{ \Theta_{1/2} } \lesssim \normC{ \Theta_0} + \normC{ \Theta_1} \lesssim h_{\T}\mathfrak{A}(f,g,h) + (\dt)^2 \Xi(f,g,h),
  \label{eq:estimateTheta12}
\end{equation}
where $\mathfrak{A}(f,g,h)$ and $\Xi(f,g,h)$ are defined by \eqref{eq:frakA} and \eqref{eq:Sigmanew}, respectively, and the
the hidden constants are independent of $g$, $h$, $(V_0,V_1)$, $(\widetilde V_0,\widetilde V_1)$, and $h_{\T}$
\end{lemma}
\begin{proof}
The proof of \eqref{eq:estimateTheta12} follows directly from the estimates \eqref{eq:V0-tildeV0} and \eqref{eq:V1-tildeV1}. In what follows we derive \eqref{eq:estimatefrakd12}. To accomplish this task and simplify notation, we define
\begin{equation}
 \mathfrak{D} (\mathcal{U}) : =  \mathcal{U}(0) + \dt \partial_t  \mathcal {U}(0) + \frac{1}{2}(\dt)^2 \partial_t^2  \mathcal {U}(0).
\end{equation}
Now, notice that
\begin{multline*}
\| \tr(\Theta_1 - \Theta_0 ) \|_{L^2(\Omega)} = 
\| \tr( \widetilde V_1 - G_{\T_{\Y}}\mathcal{U}(t_1) - \widetilde V_0 + G_{\T_{\Y}}\mathcal{U}(0) ) \|_{L^2(\Omega)} 
 \\
 \leq \| \tr( \widetilde V_1 -  \mathfrak{D} ( G_{\T_{\Y}} \mathcal{U}) - \widetilde V_0 + G_{\T_{\Y}}\mathcal{U}(0) ) \|_{L^2(\Omega)} + 
 \| \tr(\mathfrak{D} ( G_{\T_{\Y}} \mathcal{U}) - G_{\T_{\Y}}\mathcal{U}(t_1) ) \|_{L^2(\Omega)}.
\end{multline*}

The trace estimate \eqref{Trace_estimate}, the stability property \eqref{eq:G_stable} and an application of Taylor's Theorem reveal that
\begin{align*}
  \| \tr(\mathfrak{D} ( G_{\T_{\Y}} \mathcal{U}) - G_{\T_{\Y}}\mathcal{U}(t_1) ) \|_{L^2(\Omega)} & \lesssim \normC{\mathfrak{D} ( G_{\T_{\Y}} \mathcal{U}) - G_{\T_{\Y}}\mathcal{U}(t_1)} \lesssim  \normC{\mathfrak{D} ( \mathcal
  {U}) - \mathcal{U}(t_1)}
  \\
  & \lesssim (\dt)^3 \| \partial_t^3 \nabla \mathcal{U}(\cdot,z) \|_{L^{\infty}(0,T;L^2(y^{\alpha},\C))}.
\end{align*}

We now use the definitions of $\widetilde V_0$ and $\widetilde V_1$ to arrive at
\begin{multline}
\| \tr( \widetilde V_1 -  \mathfrak{D} ( G_{\T_{\Y}} \mathcal{U}) - \widetilde V_0 + G_{\T_{\Y}}\mathcal{U}(0) ) \|_{L^2(\Omega)} 
\leq (\dt) \|\Pi_{x'}h - \tr G_{\T_{\Y}}\partial_t \mathcal{U}(0)  \|_{L^2(\Omega)} 
\\
+(\dt)^2 \| Z - \tr G_{\T_{\Y}}\partial_t^2 \mathcal{U}(0)  \|_{L^2(\Omega)}.
\end{multline}
To bound $\|\Pi_{x'}h - \tr G_{\T_{\Y}}\partial_t \mathcal{U}(0)  \|_{L^2(\Omega)}$ we proceed as follows:
\begin{multline}
\|\Pi_{x'}h - \tr G_{\T_{\Y}}\partial_t \mathcal{U}(0)  \|_{L^2(\Omega)} \leq \|\Pi_{x'}h - h  \|_{L^2(\Omega)} 
\\
+ \|\tr \partial_t \mathcal{U}(0) - \tr G_{\T_{\Y}}\partial_t \mathcal{U}(0)  \|_{L^2(\Omega)} \lesssim h_{\T}^{1+s} ( \| h\|_{\mathbb{H}^{1+s}(\Omega)} + \mathfrak{A}(f,g,h) ),
\end{multline}
where we have used \eqref{eq:tr_P} with $\ell = 1$.
Similar arguments allow us to control $\| Z - \tr G_{\T_{\Y}}\partial_t^2 \mathcal{U}(0)  \|_{L^2(\Omega)}$. This concludes the proof.
\end{proof}
\begin{remark}[influence of computable data in error estimates]\rm
If $(\widetilde V_0, \widetilde V_1)$ are used as initial data for the trapezoidal multistep method of section \ref{subsec:trapezoidal} and the leapfrog scheme of section \ref{subsec:leapfrog}, then the error estimates of Theorems \ref{lemma:error_estimate_totals} and \ref{lemma:error_estimate_totalsfrog} hold with no modifications.
\end{remark}
\section{Numerical results and implementation}\label{sec:numerics}

Let $\{ \phi_1, \dots, \phi_{\M } \}$ denote a basis of $S^\bmr((0,\Y),\calG_{\sigma}^M)$ such that $\phi_1(0) = 1$ and $\phi_j(0) = 0$ for $j > 1$.  The corresponding mass and stiffness matrices are denoted by $B_{\Y}$ and $A_{\Y}$:
\[
  \left(B_{\Y}\right)_{ij} = \int_0^\Y y^\alpha \phi_i(y)\phi_j(y) dy, \qquad
    \left(A_{\Y}\right)_{ij} = \int_0^\Y y^\alpha \phi'_i(y)\phi'_j(y) dy.
\]
We denote by $B_{\Omega}$ and $A_{\Omega}$ the standard mass and stiffness matrices corresponding to the finite element space $S^1_0(\Omega,\T)$. In what follows, we describe the implementation of a discrete Dirichlet-to-Neumann map. Once this operation is available the time--stepping methods that are proposed in this work can be implemented in a standard way; 
in the case of the implicit method some further steps may be needed in order to obtain an efficient algorithm.

\subsection{Discrete Dirichlet-to-Neumann map}

Given $U \in S^1_0(\Omega,\T)$, we consider the problem: Find $V \in \V(\T_{\Y})$ and $\eta \in S^1_0(\Omega,\T)$ such that
\begin{equation}
\label{eq:discrete_system}
\begin{split}
a_\Y(V,W) &= \langle \eta, \tr W\rangle \qquad \forall W \in \V(\T_{\Y}),\\
\tr V &= U.
\end{split}
\end{equation}
Let us denote by $\Ub$, $\Vb$, and $\etab$  the coefficient vectors associated with the discrete functions $U$, $V$ and $\eta$. Note that the first $N$ components of $\Vb$ and $\Ub$ are equal. We denote the remaining components of $\Vb$  by  $\widetilde\Vb = (\Vb)_i$, $i = N+1,\dots, N\M$; we recall that $N = \# \T$, the number of degrees of freedom of $\T$. With this notation at hand, the matrix system \eqref{eq:discrete_system} takes the form
\[
(B_\Y \otimes A_\Omega+ A_\Y \otimes B_\Omega)
\begin{pmatrix}
  \Ub\\\widetilde\Vb
\end{pmatrix}
 =
\begin{pmatrix}
  B_\Omega \etab \\ \mathbf{0}
\end{pmatrix}.
\]

We denote by $\widetilde B_\Y$ and $\widetilde A_\Y$ the matrices that are obtained by removing the first row and first column from $B_\Y$ and $A_\Y$, respectively. Let $\tilde b_\Y$ and $\tilde a_\Y$ denote the vectors containing the first components of the rows $i = 2,\dots,\M$ of the matrices $B_\Y$ and $A_\Y$. 

The vector $\widetilde\Vb$ is the solution of 
\begin{equation}
  \label{eq:DtN_sys}
(\widetilde B_\Y \otimes A_\Omega+ \widetilde A_\Y \otimes B_\Omega)
\widetilde\Vb
 =
-\left(\tilde b_\Y \otimes A_\Omega+ \tilde a_\Y \otimes B_\Omega \right)\Ub.  
\end{equation}
Once this system is solved we can obtain $\etab$ by solving
\[
B_\Omega \etab = (bA_\Omega +aB_\Omega) \Ub+(\tilde b_\Y^T \otimes A_\Omega+\tilde a_\Y^T \otimes B_\Omega) \widetilde\Vb,
\]
where $b = (B_\Y)_{11}$ and $a = (A_\Y)_{11}$. We denote by $L_h^s$ the matrix that describes the linear map $u \mapsto \eta$:
\begin{equation}
  \label{eq:DtN_matrix}
L_h^s \Ub = B_\Omega\etab.  
\end{equation}

The main computational cost when solving \eqref{eq:discrete_system} is the resolution of \eqref{eq:DtN_sys}. This can be done efficiently by diagonalizing the small matrices $\widetilde B_\Y$ and $\widetilde A_\Y$. In fact, since both matrices are
symmetric and positive definite, we can find a matrix of generalized eigenvectors $X$ such that
\[
X^T \widetilde B_\Y X = \diag(\mu_1,\dots,\mu_\M), \qquad
X^T\widetilde A_\Y X = \diag(1,\dots,1),
\]
where $\mu_j > 0$ denote the eigenvalues of the generalized eigenvalue problem $\widetilde B_{\Y} X = \widetilde A_{\Y} X\diag(\mu_1,\dots,\mu_\M)$. The system \eqref{eq:DtN_sys} can thus be transformed to $\M$ independent linear systems
\[
(\mu_j A_\Omega+ B_\Omega)
\hat\Vb_j
 =
\left( (X^T\tilde b_\Y)_j A_\Omega+ (X^T\tilde a_\Y)_j B_\Omega \right)\Ub,
\]
where $\hat \Vb^T = (\hat \Vb^T_1,\dots, \hat \Vb^T_\M)$,
\[
\widetilde \Vb = (X \otimes I) \hat \Vb,
\]
and $I \in \mathbb{R}^{N \times N}$ is the identity matrix.

\subsection{Leapfrog time-stepping scheme}

Using the previously defined operator $L_h^s$, the leapfrog scheme can be now written in the familiar form
\[
\frac1{\dt^2}B_\Omega (\Ub_{k+1}-2\Ub_{k}+\Ub_{k-1})+L_h^s \Ub_k = B_\Omega\fb_k,
\]
where $\fb_k$ is the coefficient vector containing the $L^2$ projection of $f(t_k)$ onto the space $S^1_0(\Omega,\T)$. The main cost is the application of $L_h^s \Ub_k$ in each step followed by the inversion of the mass matrix $B_\Omega$; the latter being usually cheap.

\subsection{Trapezodial  time-stepping scheme}
The matrix system can be written, in a familiar form, involving only functions in $S^1_0(\Omega,\T)$:
\[
\frac1{\dt^2}B_\Omega (\Ub_{k+1}-2\Ub_{k}+\Ub_{k-1})+\frac14 L_h^s(\Ub_{k+1}+2\Ub_{k}+\Ub_{k-1})  = \frakc B_\Omega\fb_k.
\]
The difficulty now is that at each step we need to solve the system
\[
\left(B_\Omega+\tfrac{\dt^2}{4}L_h^s\right)\Ub_{k+1} = \tilde \fb_k,
\]
where $\tilde \fb_k$ contains known terms. While we could solve this system iteratively, it is more efficient to unwrap again the operator $L_h^s$ to see that $\Ub_{k+1}$ satisfies the system
\[
(B_\Y \otimes A_\Omega+ A_\Y \otimes B_\Omega)
\begin{pmatrix}
  \Ub_{k+1}\\\widetilde\Vb_{k+1}
\end{pmatrix}
 =
\begin{pmatrix}
  B_\Omega\etab_{k+1} \\ 0
\end{pmatrix}
\]
and
\[
B_\Omega \Ub_{k+1}+\tfrac{\dt^2}{4} B_\Omega\etab_{k+1} = \tilde \fb_k.
\]
Denoting by $E_1 = \diag(1,0,\dots,0)$, we can write this as a single system
\[
\left(\tfrac{\dt^2}4 B_\Y \otimes A_\Omega+ (E_1+\tfrac{\dt^2}4 A_\Y) \otimes B_\Omega\right)
\begin{pmatrix}
  \Ub_{k+1}\\\widetilde\Vb_{k+1}
\end{pmatrix}
 =
\begin{pmatrix}
  \tilde \fb_{k} \\ 0
\end{pmatrix}.
\]
As the matrices $B_\Y$ and $E_1+\tfrac{\dt^2}4 A_\Y$ are again symmetric and positive definite similar diagonalization procedure can be used to solve the system efficiently.

\subsection{Numerical results}

\subsubsection{A 1D example}
Let us first perform a numerical example for $n=1$. Let $\Omega = (0,1)$ and consider the space--fractional wave equation
\[
  \partial_t^2 u +(-\Delta )^s u = f, 
\]
with $f(x',t) = (\pi^{2s}-1)\sin (t) \sin (\pi x')$. The initial data are such that the exact solution is
$u(x',t) = \sin (t) \sin (\pi x')$. The initial values for both time--stepping schemes can be taken as $U_0 = 0$ and $U_1 = \dt \Pi_{x'} \sin (\pi x')$. Note that the additional $O(\dt^2)$ term that is needed
in the definition of $U_1$ is zero because $u''(0) = 0$.

We set the final time as $T = \pi/2$ and perform numerical experiments for the following choices of the parameter $s$: $s = 1/4$ and $s = 3/4$. For space discretization in $\Omega$, we consider a uniform mesh with meshwidth $h_\T$. For the trapezoidal scheme, and to obtain linear convergence, we set $\dt = (0.5 h_{\T})^{1/2}$. For the leapfrog scheme, in order to ensure the stability of the scheme, we choose $\dt = (0.5 h_{\T})^{\max(1/2,s)}$. Notice that, the aforementioned choices of the parameter $s$ would render the leapfrog scheme unstable as a solution technique for the standard wave equation.

In Figure~\ref{fig:oned_example} we show the experimental rate of convergence for
\begin{equation}
  \label{eq:1d_error}
\text{error} =   \|U_K - u(T)\|_{\Hs}.  
\end{equation}
We observe that, as expected, the error decays linearly with respect to $h_{\T}$. We also notice that the errors for both discretization schemes are almost identical.

\begin{remark}[properties of the leapfrog scheme]\rm
When compared with the trapezoidal scheme, the leapfrog scheme is easier to implement. However, it seems to lose one of the main advantages that it has for the resolution of the standard wave equation. As the diffusion operator is nonlocal the explicit nature of the scheme is no longer of such an importance.
\end{remark}

\begin{figure}
  \centering
  \includegraphics[width=0.45\textwidth]{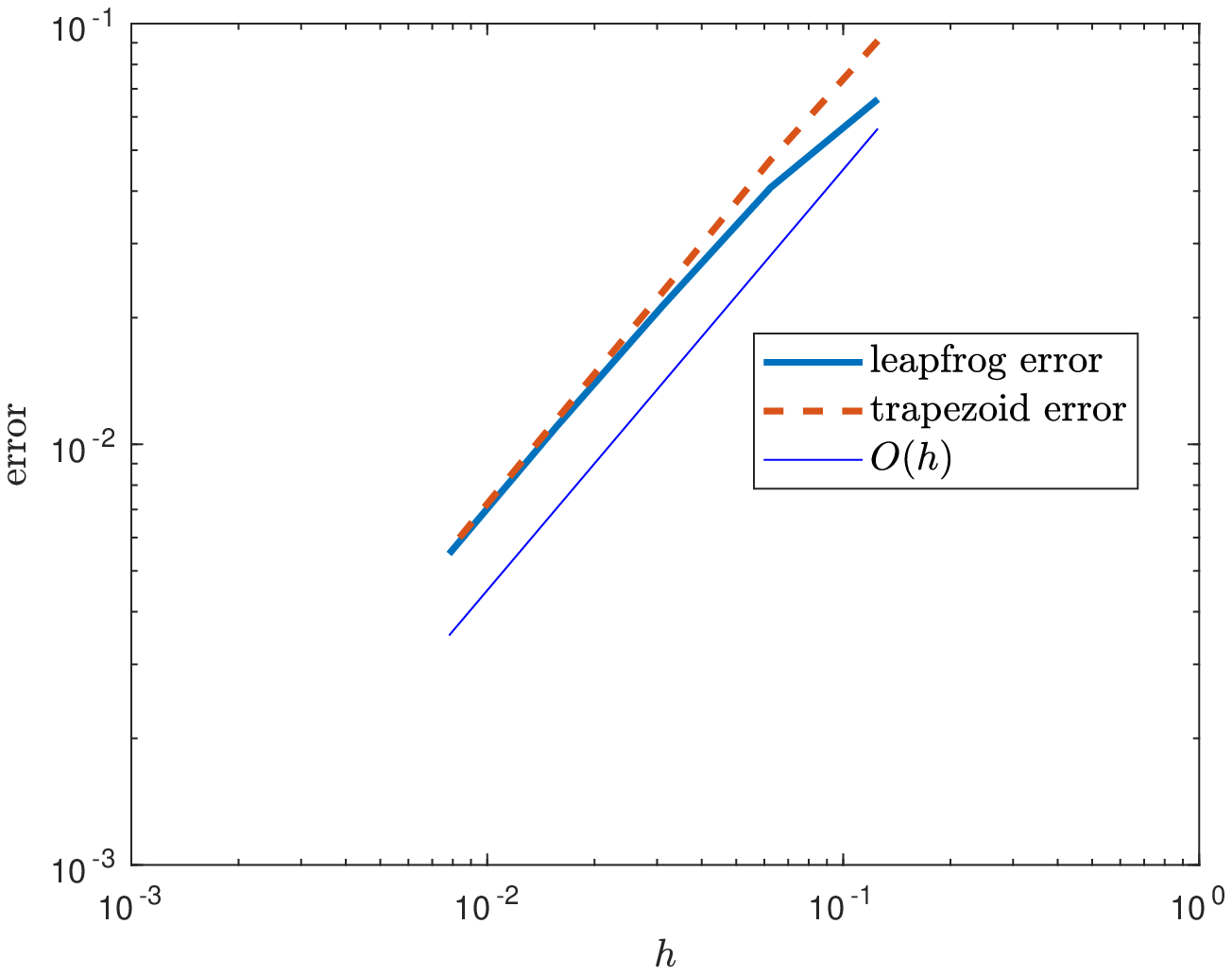}
    \includegraphics[width=0.45\textwidth]{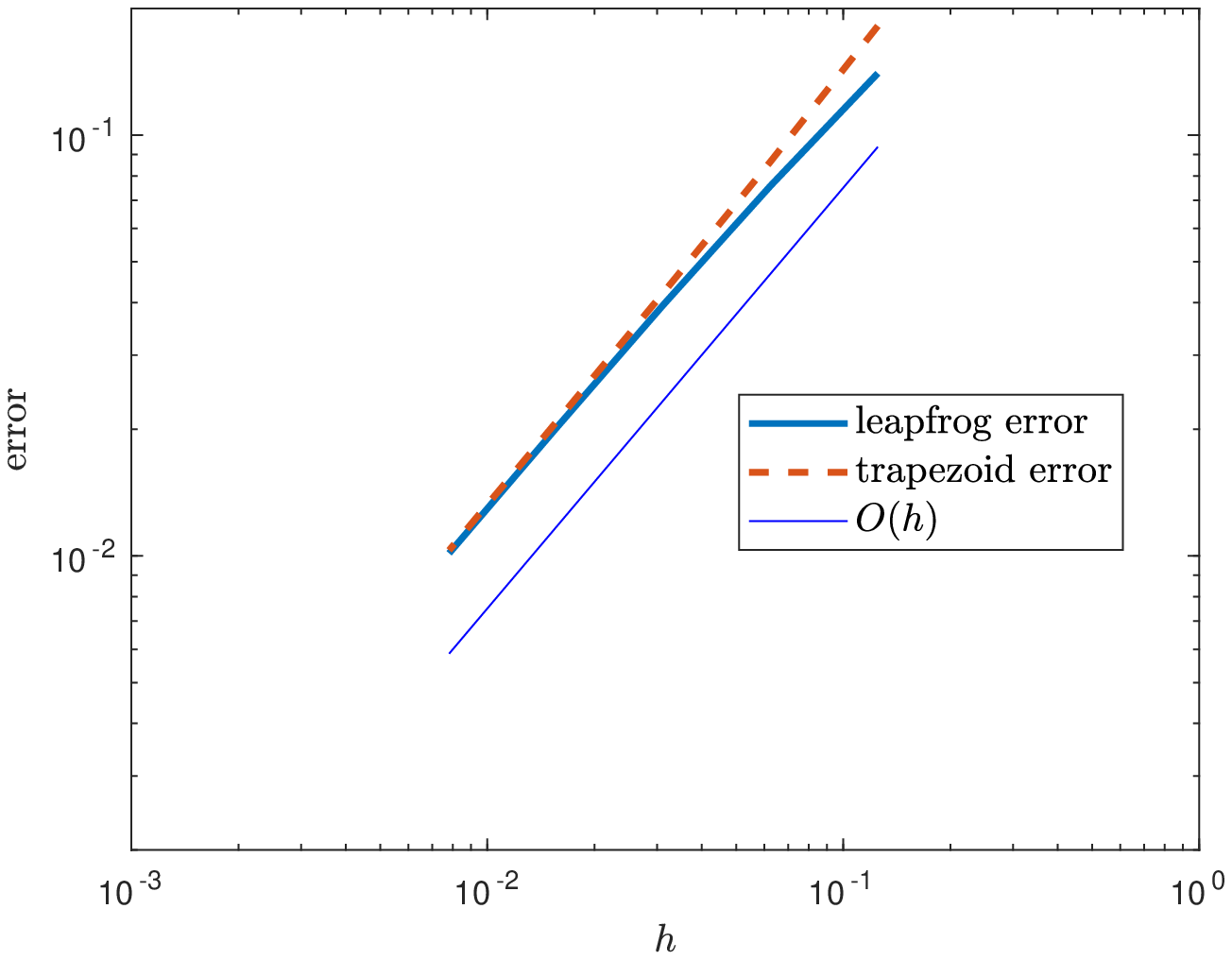}
  \caption{Experimental rates of convergence for the error \eqref{eq:1d_error} for a one dimensional example. The experiment for $s = 1/4$ is shown on the left and the one for $s = 3/4$ on the right.}
  \label{fig:oned_example}
\end{figure}
\subsubsection{A 2D example}

We let $n= 2$, and consider the square domain $\Omega = (-1,1)^2 \subset \mathbb{R}^2$ and the space--fractional wave equation
\[
  \partial_t^2 u +(-\Delta )^s u = 0.
\]
The exact solution is given by
\[
  u(x',t) = \cos(2^{s/2}\pi^s t) \sin(\pi x'_1) \sin(\pi x_2').
\]
We set  $T = 1.5$ and $\dt = 0.5h_\T$. In this experiment we measure, in the $L^2(\Omega)$--norm,  the error committed in the approximation of the time--derivative:
\begin{equation}
  \label{eq:2d_error}
\text{error}_{2D} =   \|\frakd U_K - \partial_t u(t_{K-1/2})\|_{L^2(\Omega)}.
\end{equation}
The computations were done using the NGSolve/Netgen software package \cite{netgen1,netgen2}.

The convergence properties of the leapfrog scheme are presented in Figure~\ref{fig:twod_example}. It can be observed that, for both values of the parameter $s$ considered, the experimental rate of convergence for the $\text{error}_{2D}$ decays quadratically with respect to $h_{\T}$. We notice that the observed rates are better than the ones derived in Theorem \ref{lemma:error_estimate_totalsfrog}, but are in agreement with approximation theory since, in this case, $u \in H^2(\Omega)$. The improvement of the estimate \eqref{eq:error_estimate_total1sfrog} of Theorem \ref{lemma:error_estimate_totalsfrog} from $h_{\T}^{1+s}$ to $h_{\T}^2$ is an open problem.

\begin{figure}
  \centering
  \includegraphics[width=0.65\textwidth]{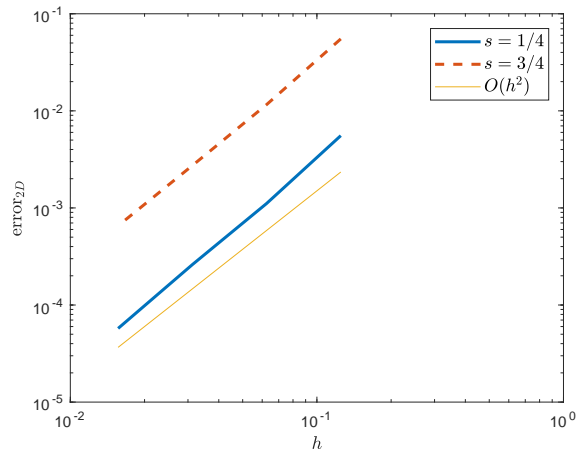}
  \caption{Experimental rates of convergence for the error \eqref{eq:2d_error} when $\Omega = (-1,1)^2$ and for different values of $s$. }
  \label{fig:twod_example}
\end{figure}

\bibliographystyle{plain}
\bibliography{biblio}

\end{document}